\documentclass{amsproc}
\usepackage{amsmath, amsthm, amssymb,fullpage}
\usepackage{amsmath}
\usepackage{amsfonts}
\usepackage{amssymb}
\usepackage{amsthm}
\usepackage[utf8]{inputenc}
\usepackage{breqn}
\usepackage{afterpage}
\usepackage{longtable}
\usepackage{indentfirst}
\usepackage{caption}
\usepackage{subcaption}
\usepackage{graphicx}
\graphicspath{ {./images/} }
\usepackage{color}
\usepackage{xcolor}

\newtheorem{theorem}{Theorem}[section]
\newtheorem{lemma}[theorem]{Lemma}
\newtheorem{corollary}[theorem]{Corollary}
\newtheorem{proposition}[theorem]{Proposition}

\theoremstyle{definition}
\newtheorem{definition}[theorem]{Definition}
\newtheorem{example}[theorem]{Example}

\theoremstyle{remark}
\newtheorem{remark}[theorem]{Remark}

\numberwithin{equation}{section}

\begin{document}

\title{Dynamics of Fuchsian meromorphic connections with real periods}

\author{Marco Abate}
\author{Karim Rakhimov}
\address{Department of  Mathematics, University of Pisa, Pisa, Italy, 56127}
\email{marco.abate@unipi.it}
\address{V.I. Romanovskiy Institute of Mathematics of Uzbek Academy of Sciences,  Tashkent, Uzbekistan}
\curraddr{}
\email{karimjon1705@gmail.com}
\thanks{The first author is partially supported by the French-Italian University and by Campus France through the Galileo program, under the project ``From rational to transcendental: complex dynamics and parameter spaces”, as well as by Istituto Nazionale di Alta Matematica (INdAM). The second author is partially supported by the National University of Singapore through the grant A-0004285-00-00.}

%    General info

\date{}

%\dedicatory{This paper is dedicated to our advisors.}


\begin{abstract}
In this paper, we study the dynamics of geodesics of Fuchsian meromorphic connections with real periods, giving a precise characterization of the possible $\omega$-limit sets of simple geodesics in this case. The main tools are the study of the singular flat metric associated to the meromorphic connection, an explicit description of the geodesics nearby a Fuchsian pole with real residue larger than $-1$ and a far-reaching generalization to our case of the classical Teichm\"uller lemma for quadratic differentials. 
%In particular we prove a Poincaré-Bendixson theorem for non-selfintersecting geodesics excluding some for that periodic geodesics and Fuchsian geodesic polygons cannot be the omega limit set of a non-periodic simple geodesic.}
\end{abstract}

\maketitle
\tableofcontents
\section{Introduction}

  A \textit{meromorphic connection} on a Riemann surface $S$ is a $\mathbb{C}-$linear operator $\nabla:\mathcal{M}_{TS}\to \mathcal{M}_S^1 \otimes\mathcal{M}_{TS}$, where $\mathcal{M}_{TS}$ is the sheaf of germs of meromorphic sections of the tangent bundle $TS$ and $\mathcal{M}_S^1$ is the space of meromorphic 1-forms on $S$, satisfying the Leibniz rule
$\nabla(\tilde{f}\tilde{s})=\text{d}\tilde{f}\otimes \tilde{s} +\tilde f\nabla \tilde s$
for all $\tilde s\in \mathcal{M}_{TS}$ and $\tilde f\in \mathcal{M}_{S}$. A \textit{geodesic} for a meromorphic connection $\nabla$ is a real smooth curve $\sigma:I\to S^o$, where $I\subseteq\mathbb{R}$ is an interval and $S^o$ is the complement of the poles of $\nabla$ in $S$, satisfying the geodesic equation $\nabla_{\sigma'}\sigma'\equiv0$.  %In \cite{AT1}, the authors discovered that there is a strong relationship between the local dynamics of the time-1 map of homogeneous vector fields and the dynamics of geodesics for meromorphic connections on Riemann surfaces. 

 The Poincaré-Bendixson theorems, which provide a potential classification of $\omega$-limit sets, were investigated within the context of simple geodesics associated with meromorphic connections in the complex projective space $\mathbb{P}^1(\mathbb C)$ by the first author and Tovena in \cite{AT1}. Subsequently,  in \cite{AB}, the first author and Bianchi extended their inquiry to encompass any compact Riemann surface $S$, successfully establishing Poincaré-Bendixson theorems for simple geodesics in this more general setting. See also \cite{NST} for examples of particularly interesting $\omega$-limit sets.

 In a recent paper \cite{R1}, the second author provided a relationship between meromorphic connections and $k$-differentials and, more generally, dilation surfaces. Nowadays, meromorphic $k$-differentials and dilation surfaces are heavily studied (see, for example, \cite{BCG, DFG, G21, Sch, ST, T}); it turns out that some of the possibilities described in the quoted Poincar\`e-Bendixson theorems do not occur in the case of $k$-differentials. Since $k$-differential can be seen a particular case of singular flat metric (see \cite{SK}), this suggested us to study Poincar\`e-Bendixson theorems for meromorphic connection with \textit{real periods}, i.e., meromorphic connections sharing geodesics and singular points with a singular flat metric. Indeed, in this paper we study possible $\omega$-limit sets for simple geodesics for meromorphic connections with real periods.
 
 We require some definitions before presenting our main result (see Section \ref{holomorphic connection} for more details).  Let $\{(U_\alpha,z_\alpha)\}$ be an atlas for $S$, and $\nabla$ a meromorphic connection on $S$. By definition, there exists $\eta_\alpha\in \mathcal{M}_S^1(U_\alpha)$, such that $\nabla(\partial_\alpha)=\eta_\alpha\otimes\partial_\alpha,$
where $\partial_\alpha:=\frac{\partial}{\partial z_\alpha}$ is the induced local generator of $TS$ over $U_\alpha$.  We say that $p\in S$ is a \textit{Fuchsian} pole of a meromorphic connection $\nabla$ if for some (and hence any) chart $(U_\alpha,z_\alpha)$ around $p$ the representation of $\nabla$ has a simple pole at $p$; the \emph{residue} $\operatorname{Res}_p\nabla$ of $\nabla$ at $p$ is the residue of~$\eta_\alpha$ at~$p$. If all poles of $\nabla$ are Fuchsian then we say that $\nabla$ is a \emph{Fuchsian meromorphic connection}. A \emph{saddle connection} is a simple geodesic connecting two poles. A set $W\subset S$ with $\mathring W=\emptyset$ is a \emph{transversally Cantor-like geodesic set} if the following  conditions holds:
  \begin{enumerate}
        \item there exists a maximal non self-intersecting geodesic $\sigma\colon(\varepsilon_-,\varepsilon_+)\to S^o$ such that $W$ is the closure of the support of~$\sigma$;
         \item for any non self-intersecting geodesic $\gamma\colon(-\delta,\delta)\to S^0$ transverse to $\sigma$ the intersection $\gamma|_{[-\delta/2,\delta/2]}\cap W$ is a perfect totally disconnected set (a Cantor set).
\end{enumerate}
We can now state our main result.

\begin{theorem}\label{pbt1}
   Let $\nabla$ be a Fuchsian meromorphic connection on a compact Riemann surface $S$ with real periods. Set $S^o:=S\setminus\Sigma$, where $\Sigma$ is the set of poles of $\nabla$. Let $\sigma:[0,\varepsilon)\to S^0$ be a  maximal simple geodesic for $\nabla$. Then either
   \begin{enumerate}
     \item the $\omega$-limit set of $\sigma$ is a pole of $\nabla$; or
     \item $\sigma$ is a periodic geodesic; or
      \item the $\omega$-limit is a transversally Cantor-like geodesic set; or
     \item the $\omega$-limit set of $\sigma$  has non-empty interior and non-empty boundary, and each component of its boundary is a graph of saddle connections with no spikes and at least one pole; or
     \item the $\omega$-limit set of $\sigma$ is $S$.
   \end{enumerate}
 \end{theorem}

Notice that case (3) was missing from \cite{AT1} and \cite{AB}; in Appendix~\ref{s:appA} we clarify why it should be included, studying geodesics of meromorphic connections whose $\omega$-limit set coincides with its closure. 

The proof depends on two main ingredients. The first one is a detailed study of the behavior of geodesics around Fuchsian poles. In \cite{R1}, the second author has shown that if $p$ is a Fuchsian pole with $\operatorname{Re}\operatorname{Res}_p\nabla\le -1$ then the $\omega$-limit set of a geodesic accumulating~$p$ must reduce to $p$. So we concentrate on the case $\operatorname{Re}\operatorname{Res}_p\nabla>-1$, when we can find local coordinates giving a particularly simple local expression for the local representation $\eta_\alpha$ of~$\nabla$ (see Lemma~\ref{l8}). In particular, when the residue is real we then have explicit formulas for the geodesics close to the pole (see Proposition~\ref{prop2}). 

The second main ingredient is a formula (Theorem~\ref{t:attf}) relating the internal angles of a graph of saddle connections between Fuchsian poles with real residues larger than $-1$ bounding a domain $P\subset S$ with the values of the residues of the poles inside $P$ and the topology of $P$. This formula is a far-reaching generalization of the classical Teichm\"uller lemma for quadratic differentials (see, e.g., \cite[Theorem 14.1]{SK}); particular cases were already obtained in \cite{AB} and \cite{AT1}.

The paper is organized as follows. In Section \ref{holomorphic connection} we recall some definitions and theorems. In Section  \ref{localbehaviorfuchsianpolessec} we study the local behaviour of geodesics around Fuchsian poles. In Section \ref{ATTsection}
we prove our generalization of Teichm\"uller lemma. Finally, in Section \ref{poincare-bendexsonsection} we prove Theorem \ref{pbt1} and in Appendix~\ref{s:appA} we discuss transversally Cantor-like geodesic sets. 

\subsection*{Acknowledgement}
The authors would like to thank Xavier Buff, Fabrizio Bianchi and Guillaume Tahar for useful comments and discussions. 
\section[Preliminary notions]{Preliminary notions} \label{holomorphic connection}

In this section we recall some definitions and theorems from \cite{AT1}, \cite{AB} and \cite{R1}.

\begin{definition}
 A \textit{meromorphic connection} on the tangent bundle $TS$ of a Riemann surface $S$ is a $\mathbb{C}$-linear map   $\nabla:\mathcal{M}_{TS}\to \mathcal{M}_S^1 \otimes\mathcal{M}_{TS}$  satisfying the Leibniz rule
\[
\nabla(\tilde{f}\tilde{s})=\text{d}\tilde{f}\otimes \tilde{s} +\tilde{f}\nabla \tilde{s}
\]
for all $\tilde{s}\in \mathcal{M}_{TS}$ and $\tilde{f}\in \mathcal{M}_{S}$, where $\mathcal{M}_{TS}$ denotes the sheaf of germs of meromorphic sections of $TS$, while $\mathcal{M}_S$ is the sheaf of germs of meromorphic functions and $\mathcal{M}_S^1$ is the sheaf of meromorphic 1-forms on $S$.
\end{definition}

Let $\{(U_\alpha,z_\alpha)\}$ be an atlas for $S$ and $\nabla$ a meromorphic connection on $S$. By definition, there exists $\eta_\alpha\in \mathcal{M}_S^1(U_\alpha)$, such that
\[
\nabla(\partial_\alpha)=\eta_\alpha\otimes\partial_\alpha,
\]
where $\partial_\alpha:=\frac{\partial}{\partial z_\alpha}$ is the induced local generator of $TS$ over $U_\alpha$.

\begin{definition}
We say  that $\eta_\alpha$ is the \emph{local representation} of $\nabla$ on $U_\alpha$.\end{definition}

Let $\{\xi_{\alpha\beta}\}$ be the cocycle
representing the cohomology class $\xi\in H^1(S,\mathcal{O}^*)$ of $TS$ (hence $\xi_{\alpha\beta}=\frac{\partial z_\alpha}{\partial z_\beta}$). Over $U_\alpha\cap U_\beta$ we have
\[
\partial_\beta=\xi_{\alpha\beta}\partial_\alpha
\]
and thus
\[
\nabla(\partial_\beta)=\nabla(\xi_{\alpha\beta}\partial_\alpha)\quad\Longleftrightarrow\quad \eta_\beta\otimes \partial_\beta=\xi_{\alpha\beta}\eta_\alpha\otimes \partial_\alpha+d \xi_{\alpha\beta}\otimes \partial_\alpha,
\]
that happens if and only if
\begin{equation}\label{merrepresentation}
 \eta_\beta=\eta_\alpha+\frac{1}{\xi_{\alpha\beta}}\partial \xi_{\alpha\beta}
\end{equation}
on $U_\alpha\cap U_\beta$. In particular, $\eta_\alpha$ and $\eta_\beta$ have the same poles on $U_\alpha\cap U_\beta$.

\begin{definition}\label{poles}
A point $p\in U_\alpha$ is a \emph{pole} of the meromorphic connection~$\nabla$ if it is a pole of a local representation~$\eta_\alpha$ of~$\nabla$ on~$U_\alpha$.
We shall denote by $\operatorname{Sing}(\nabla)\subset S$ the set of poles of~$\nabla$ and by $S^o=S\setminus\operatorname{Sing}(\nabla)$ the set of regular points of~$\nabla$. We shall say that $\nabla$ is a \emph{holomorphic connection} if it has no poles. Clearly, the restriction of $\nabla$ to~$S^o$ is holomorphic.
\end{definition}

Let $\nabla$ be a holomorphic connection (e.g., the restriction of a meromorphic connection to the set of its regular points). 
Up to shrinking, if necessary, the open sets $U_\alpha$, we can find holomorphic functions $K_\alpha \in \mathcal{O}(U_\alpha)$ such that $\eta_\alpha = \partial K_\alpha$ on $U_\alpha$. Then \eqref{merrepresentation} implies that on $U_\alpha$ and $U_\beta$:
\begin{equation}\label{mer2}
  \hat{\xi}_{\alpha\beta}=\frac{\exp(K_\alpha)}{\exp(K_\beta)}\xi_{\alpha\beta}.
\end{equation}
is a complex non-zero constant defining a cocycle $\hat\xi\in H^1(S,\mathbb{C}^*)$ representing $\xi$. 

\begin{definition}
The homomorphism $\rho\colon\pi_1(S)\to\mathbb{C}^*$ corresponding to the class $\hat{\xi}$ under the canonical isomorphism 
\[
H^1(S,\mathbb{C}^*)\cong\text{Hom}(H_1(S,\mathbb{Z}),\mathbb{C}^*)=\text{ Hom}(\pi_1(S),\mathbb{C}^*)
\]
is the \emph{monodromy representation} of the holomorphic connection $\nabla$. We say that $\nabla$ has \emph{monodromy in} a multiplicative subgroup $G$ of $\mathbb{C}^*$, or that $G$ is a \emph{monodromy group} for~$\nabla$, if the image of $\rho$ is contained within~$G$. In other words, $\hat{\xi}$ is the image of a class in $H^1(S,G)$ under the natural inclusion $G\hookrightarrow \mathbb{C}^*$. Furthermore, we say that $\nabla$ has  \emph{real periods} if it has monodromy in $S^1$.
 \end{definition}
  
 Let us now define the notion of geodesic for meromorphic connections.
 
\begin{definition}
A \textit{geodesic} for a meromorphic connection $\nabla$ on $TS$ is a real curve $\sigma\colon I\to S^o$, where $I\subseteq \mathbb{R}$ an interval, such that $\nabla_{\sigma'}\sigma'\equiv 0.$
\end{definition}

In the rest of the paper a geodesic and its support will be denoted by the same symbol if there is no risk of confusion.

Let  $\{(U_\alpha,z_\alpha)\}$ be an atlas for $S^0$ and  $\sigma\colon I\to U_\alpha$  a smooth curve. Then $\sigma$ is a geodesic for a meromorphic connection $\nabla$ if and only if
\begin{equation}\label{geodesicequation}
 (z_\alpha\circ\sigma)''+(f_\alpha\circ\sigma){(z_\alpha\circ\sigma)'}^2\equiv0
\end{equation}
where $\eta_\alpha=f_\alpha dz_\alpha$ is the local representation of $\nabla$ on $U_\alpha$.

\begin{definition} 
The \emph{residue} $\operatorname{Res}_p\nabla$ of a meromorphic connection $\nabla$ at a point $p\in S$ is the residue at $p$ of any local representation $\eta_\alpha$ of $\nabla$ on a local chart $(U_\alpha,z_\alpha)$. Thanks to \eqref{merrepresentation}, it does not depend on the chosen local representation.
\end{definition}

 In compact Riemann surfaces the sum of residues is the same for all meromorphic connections and it is given by $2g-2$, where $g$ is the genus of the surface (see, for example, \cite{IYa}).

 \begin{remark}\label{perres}
 If $\gamma\in H_1(S^o,\mathbb{Z})$ is represented by a small loop around a pole $p\in S$, then (see \cite[Proposition~3.6]{AT1})
 \[
 \rho(\gamma)=\exp\left(\int_\gamma \eta\right)\;,
 \]
 where $\eta$ is the local representation of $\nabla$ in a local chart at~$p$. In particular, since
 \[
 \operatorname{Res}_p\nabla=\frac{1}{2\pi i}\int_\gamma \eta\;,
 \]
 if $\nabla$ has real periods then all the residues of~$\nabla$ are real numbers. 
 \end{remark}

\begin{definition} 
We say that $p\in S$ is a \textit{Fuchsian} pole of a meromorphic connection $\nabla$ if for some (and hence any) chart $(U_\alpha,z_\alpha)$ around $p$ the local representation of $\nabla$ has a simple pole at $p$. If all poles of $\nabla$ are Fuchsian then we say that $\nabla$ is a \emph{Fuchsian meromorphic connection}. A Fuchsian pole is \emph{resonant} if its residue is a negative integer strictly less than~$-1$. 
\end{definition}

\subsection{Local isometries and $\nabla$-charts.}
In this subsection we recall some properties of local isometries and $\nabla$-charts; we shall often use these notions in the rest of the paper.

\begin{definition}\label{d:isometry}
Let $\nabla$ be a holomorphic connection on a Riemann surface $S$. Given a simply connected chart $(U_\alpha,z_\alpha)$ for $S$, let $\eta_\alpha$ be the local representation of $\nabla$ on $U_\alpha$ and let $K_\alpha\colon U_\alpha\to\mathbb{C}$ be a holomorphic primitive of $\eta_\alpha$. A \emph{local isometry} of $\nabla$ on $U_\alpha$ is a holomorphic primitive $J_\alpha:U_\alpha\to\mathbb{C}$ of $\exp(K_\alpha)$ on $U_\alpha$. 
 \end{definition}

\begin{proposition}\label{p:localis}
Let $\nabla$ be a holomorphic connection on a Riemann surface $S$. Let $\{(U_\alpha,z_\alpha)\}$ be an atlas for $S$ with simply connected charts. If $J_\alpha$ is a local isometry of $\nabla$ on $U_\alpha$ then
\begin{enumerate}
\item if $\eta_\alpha\equiv0$ then $J_\alpha=az_\alpha+b$ for some $a\in \mathbb{C}^*$ and $b\in\mathbb{C}$; 
\item if $\eta_\alpha\equiv 0$ and $(U_\beta,z_\beta)$ is another chart with $\eta_\beta\equiv0$ and $U_\alpha\cap U_\beta$ is connected, then there exists a local isometry $J_\beta$ of $\nabla$ on $U_\beta$ such that $J_\beta|_{U_\alpha\cap U_\beta}\equiv J_\alpha|_{U_\alpha\cap U_\beta}$;
\item if $\sigma$ is a geodesic for $\nabla$ on $U_\alpha$ then $J_\alpha(\sigma)$ is a Euclidean segment in $z_\alpha(U_\alpha)$;
\item if $\tilde{J}_\alpha$ is another local isometry of $\nabla$ on $U_\alpha$ then $\tilde{J}_\alpha= aJ_\alpha+b$ for some $a\in \mathbb{C}^*$ and $b\in\mathbb{C}$;
\end{enumerate}
\end{proposition}

\begin{proof}
(1) If $\eta_\alpha\equiv0$ then $K_\alpha$ is constant and hence $J_\alpha=az_\alpha+b$ for some $a\in \mathbb{C}^*$ and $b\in\mathbb{C}$. 

(2) Fix $K_\alpha\in\mathbb{C}$ such that $J_\alpha$ is a holomorphic primitive of $\exp(K_\alpha)$. Since $\eta_\beta\equiv0$, we can choose $K_\beta\in\mathbb{C}$ so that it coincides with $K_\alpha$ on the connected set $U_\alpha\cap U_\beta$; then choose
$J_\beta$ so that it coincides with $J_\alpha$ in one point of $U_\alpha\cap U_\beta$. It follows that $J_\beta|_{U_\alpha\cap U_\beta}\equiv J_\alpha|_{U_\alpha\cap U_\beta}$.

(3) See \cite[Proposition 2.2]{AB}.

(4) It is enough to show that $\tilde{J}_\alpha'=a{J}_\alpha'$ with $a\in \mathbb{C}^*$. By definition $\tilde{J}_\alpha'=\exp{\tilde{K}_\alpha}$, where $\tilde{K}_\alpha$ is a holomorphic primitive of $\eta_\alpha$. Since $K_\alpha$ too is a holomorphic primitive of $\eta_\alpha$ there exists $c\in \mathbb{C}$ such that $\tilde{K}_\alpha={K}_\alpha+c$. So $\tilde{J}_\alpha'=\exp{\tilde{K}_\alpha}=e^c\exp{K_\alpha}=e^c {J}_\alpha'$.
\end{proof}

\begin{definition}\label{nablaatlasdef} 
Let $\nabla$ be a holomorphic connection on a Riemann surface $S$.
A simply connected chart $(U_\alpha,z_\alpha)$  is said to be a $\nabla-$\emph{chart} if the local representation of $\nabla$ on $U_\alpha$ is identically zero.
A  \emph{$\nabla-$atlas} is an atlas $\{(U_\alpha,z_\alpha)\}$ for $S$ such that all charts are $\nabla-$charts. A \emph{Leray atlas} for $S$  is a simply connected atlas $\{(U_\alpha,z_\alpha)\}$ such that intersection of any two charts of the atlas is  simply connected or empty. 
 \end{definition}

In \cite{R1} it is shown that it is always possible to find a Leray $\nabla$-atlas for any Riemann surface $S$ with a holomorphic connection $\nabla$. Proposition \ref{p:localis} imply that the representation of geodesics on $\nabla$-charts is given by Euclidean segments.

\begin{lemma}[{See \cite[Lemma 2.8]{R1}}] \label{monodromy}
Let $\nabla$ be a holomorphic connection on a Riemann surface $S$ having monodromy in $G$, a multiplicative subgroup of $\mathbb{C}^*$. Then there exists a Leray $\nabla-$atlas $\{(U_\alpha,z_\alpha)\}$ for $S$ such that the changes of coordinates in the atlas have the form $z_\beta= a_{\alpha\beta} z_\alpha+c_{\alpha\beta}$ on $U_\alpha\cap U_\beta$, where $a_{\alpha\beta}\in G$ and $c_{\alpha\beta}\in\mathbb{C}$. 
\end{lemma} 

\begin{definition} \label{def:G-atlas}
    A Leary $\nabla$-atlas with the property described in the previous lemma will be called a \emph{Leray $G$-atlas} for~$\nabla$. 
\end{definition}

\subsection{Poincar\`e-Bendixson theorems}\label{secpoincarebend}
In this subsection we recall a Poincar\`e-Bendixson theorem for meromorphic connections on compact Riemann surfaces, giving a classification of the possible $\omega$-limit sets for the geodesics of meromorphic connections on compact Riemann surfaces. 

\begin{definition}
Let $\sigma\colon(\varepsilon_-,\varepsilon_+)\to S$ be a curve in a Riemann surface $S$. Then the \emph{$\omega$-limit set} of~$\sigma$ is given by the points $p\in S$ such that there exists a sequence $\{t_n\}$, with $t_n\uparrow \varepsilon_+$, such that $\sigma(t_n)\to p$. 
\end{definition}

\begin{definition}
 A geodesic $\sigma\colon[0,l]\to S$ is \emph{closed} if $\sigma(l)=\sigma(0)$ and $\sigma'(l)$ is a positive multiple of~$\sigma'(0)$; it is \emph{periodic} if $\sigma(l)=\sigma(0)$  and $\sigma'(l)=\sigma'(0)$.
\end{definition}

\begin{definition}
 A \emph{saddle connection} for a meromorphic connection $\nabla$ on  $S$ is a maximal geodesic $\sigma\colon(\varepsilon_{-},\varepsilon_{+}) \to S^o$, with $\varepsilon_{-}\in [-\infty,0)$ and $\varepsilon_{+}\in (0,+\infty]$, such that $\sigma(t)$ tends to a pole of $\nabla$ both when $t\uparrow \varepsilon_{+}$ and when $t\downarrow\varepsilon_{-}$.

A \emph{graph of saddle connections} is a connected  graph in $S$ whose vertices are poles and whose arcs are disjoint simple (i.e., not self-intersecting) saddle connections. A \emph{spike} is a saddle connection of a graph which does not belong to any cycle of the graph.

A \emph{boundary graph of saddle connections}  is a graph of saddle connections which is also the boundary of a connected open subset of $S$. A boundary graph is
\emph{disconnecting} if its complement in $S$ is not connected.
\end{definition}

\begin{definition}\label{d:alldef}
Let $\nabla$ be a meromorphic connection on a compact Riemann surface  $S$. 
\begin{itemize}
\item  A \emph{geodesic (n-)cycle} is the union of $n$ simple smooth curves $\sigma_j\colon[0,1]\to S$ such that each restriction $\sigma_j|_{(0,1)}$ is a geodesic for~$\nabla$ and their supports are disjoint except for the conditions $\sigma_j(0)=\sigma_{j-1}(1)$ for  $j=2,...,n$ and $\sigma_1(0)=\sigma_n(1)$. The points $\sigma_j(0)$ are the \emph{vertices} of the geodesic cycle. We say that a geodesic cycle is \emph{regular} if all vertices of the geodesic cycle are regular points. We say a geodesic cycle is \emph{Fuchsian} if any vertex the geodesic cycle is either a Fuchsian pole or a regular point.  A (\textit{Fuchsian}) \textit{geodesic polygon} is a connected open set whose boundary is a (Fuchsian) geodesic cycle.
\item  A \emph{geodesic (m-)multicurve} is a union of $m$ disjoint geodesic cycles. A geodesic multicurve will be said to be \emph{disconnecting} if it disconnects $S$, \emph{non-disconnecting } otherwise. We say that a geodesic multicurve is \emph{Fuchsian} (\emph{regular}) if it is a union of $m$ disjoint Fuchsian (regular) geodesic cycles.
\item A \emph{part} $P$ is the closure of a connected open subset of $S$ whose boundary is a multicurve $\gamma$. A component $\sigma$ of $\gamma$ is \emph{surrounded }if the interior of $P$ contains both sides of a tubular neighbourhood of $\sigma$ in $S$; it is \emph{free} otherwise. The \emph{filling} $\widetilde{P}$ of a part $P$ is the compact surface obtained by gluing a disk along each of the free components of $\gamma$ and not removing any of the surrounded components of $\gamma$.
\item A set $W\subset S$ with $\mathring W=\emptyset$ is a \textit{transversally Cantor-like geodesic set} if the following  conditions holds:
  \begin{enumerate}
        \item there exists a maximal non self-intersecting geodesic $\sigma\colon(\varepsilon_-,\varepsilon_+)\to S^o$ such that $W$ is the closure of the support of~$\sigma$;
         \item for any simple geodesic $\gamma\colon(-\delta,\delta)\to S^0$ transverse to $\sigma$ the intersection $\gamma([-\delta/2,\delta/2])\cap W$ is a perfect totally disconnected set (a Cantor set).
  \end{enumerate}
\end{itemize}
\end{definition}

Now we state the Poincar\'e-Bendixson theorem for meromorphic connections on a compact Riemann surface $S$ proved in \cite[Theorem 4.6]{AT1} and \cite[Theorem 4.3]{AB}, with the missing case included (see Appendix~\ref{s:appA}).

 \begin{theorem}[{Abate-Bianchi-Tovena}]\label{t3}
Let $\sigma\colon[0,\varepsilon)\to S^o$ be a maximal geodesic for a meromorphic connection $\nabla$ on $S$, where
$S^o=S\setminus\{p_0,p_1,\dots,p_r\}$ and $p_0,p_1,\dots,p_r$ are the poles of $\nabla$. Then either
\begin{enumerate}

  \item $\sigma(t)$ tends to a pole of $\nabla$ as $t\to\varepsilon$; or
  \item $\sigma$ is closed; or
  \item the $\omega$-limit set of $\sigma$ is given by the support of a closed geodesic; or
  \item the $\omega$-limit set of $\sigma$ is a boundary graph of saddle connections; or
  \item the $\omega$-limit set of $\sigma$ is a transversally Cantor-like geodesic set; or
  \item the $\omega$-limit set of $\sigma$ is all of $S$; or
  \item the $\omega$-limit set of $\sigma$ has non-empty interior and non-empty boundary and each component of its boundary is a graph of saddle connections with no spikes and at least one pole; or
  \item  $\sigma$ intersects itself infinitely many times.
\end{enumerate}
 Furthermore, in cases $(3)$ or $(4)$ the support of $\sigma$ is contained in only one of the components of the complement of its $\omega$-limit set, which is a part $P$ of $S$ having the $\omega$-limit set as boundary.
\end{theorem}

 \begin{remark}\label{r:cls}
As anticipated in the introduction, case (5) has been missed in \cite[Theorem 4.3]{AB}. It appears when the $\omega$-limit set $W$ of $\sigma$ has an empty interior and 
$\mathrm{supp}(\sigma)\subseteq W$, a case not considered in the proof of \cite[Theorem 4.3]{AB}; see Appendix~\ref{s:appA}. 
\end{remark} 

\begin{remark}
All cases listed in Theorem \ref{t3} can be realized: see \cite{AB,AT1,DFG,NST}. 
\end{remark}

\subsection{Singular flat metrics, $k-$differentials and meromorphic connections}
In this section we recall the relationships among meromorphic connections, singular flat metrics and  meromorphic $k$-differentials (see \cite{AT1, R1}).

\begin{definition}
Let $S$ be a Riemann surface and $\Sigma\subset S$ a discrete set having no limit points in $S$. Set $S^o:=S\setminus\Sigma$. We say that $g$ is a \emph{singular flat metric} on $S$ if $g$ is a flat metric on $S^o$ and  for any $p\in\Sigma$ there exist $\rho_p\in \mathbb{R}$ such that, for any chart $(U_\alpha,z_\alpha)$ centered at $p$ with $U_\alpha\cap \Sigma=\{p\}$, writing $g^{\frac{1}{2}}=e^{u_\alpha}|dz_\alpha|$ on $U_\alpha\setminus\{p\}$ the function $e^{u_\alpha}$ satisfies 
\[
\lim\limits_{z_\alpha\to 0}\frac{e^{u_\alpha}}{|z_\alpha|^{\rho_p}}>0\;.
\]
Notice that saying that $g$ is flat is equivalent to saying that $u_\alpha\colon U_\alpha\setminus\{0\}\to\mathbb{R}$ is harmonic. We say that $\rho_p$ is the \emph{residue} of $g$ at the \emph{singular point} $p$ and that $\Sigma$ is the \emph{singular set} of $g$.
\end{definition}

\begin{definition} 
Let $\nabla$ be a meromorphic connection on a Riemann surface $S$ and $g$ a singular flat metric  on $S$.  We say that $g$ and $\nabla$ are \emph{adapted} to each other if they have the same singular sets and the same geodesics (intended as parametrized curves).
\end{definition}

It turns out that a meromorphic connection is adapted to a singular flat metric if and only if it is Fuchsian with real periods.

\begin{theorem}[{\cite[Proposition 1.2]{AT1}, \cite[Theorem 3.8]{R1})}]\label{t8}
Let $\nabla$ be a Fuchsian meromorphic connection on a Riemann surface $S$ and let $\Sigma$ be the set of poles of $\nabla$. Set $S^o=S\setminus\Sigma$. If $\nabla$ has real periods then there exists a singular flat metric $g$ adapted to $\nabla$. Moreover, $g$ is unique up to a positive constant multiple and the local isometries of $\nabla$ are isometries between (a multiple of)~$g$ and the Euclidean metric.

Conversely, if $g$ is a singular flat metric on $S$ with singular set $\Sigma$ then there exists a unique meromorphic connection $\nabla$ with $\Sigma$ as set  of poles such that $g$ is adapted to $\nabla$. Moreover, $\nabla$ is Fuchsian with real periods.

Furthermore, if $\rho_p$ is the residue of a singular point $p$ of $g$ then $\operatorname{Res}_p\nabla=\rho_p$ and vice versa.
\end{theorem}

\begin{remark}
As a consequence of Theorem \ref{t8}, if $p$ is a pole with real residue, we can always construct a singular flat metric adapted to $\nabla$ in a simply connected neighbourhood $U$ of $p$ not containing other poles, because $\nabla$ in $U$ has real periods (see Remark~\ref{perres}).   
\end{remark} 

The following lemma (see \cite[Corollary 4.5]{AT1} for the case $S\subseteq\mathbb C$) shows that if $\nabla$ has real periods then all closed geodesics are periodic. 

\begin{corollary}\label{l:closedperiodic}
 Let $\nabla$ be a Fuchsian meromorphic connection on a Riemann surface $S$ with real periods. Then any closed geodesic of $\nabla$ is periodic.
\end{corollary}
\begin{proof}
By Theorem~\ref{t8} (see \cite[Proposition~1.2]{AT1}) there is a singular flat metric $g$ adapted to $\nabla$. In particular, the geodesics of $\nabla$ 
are the same as the geodesics of $g$. Since closed geodesics of a Riemannian metric are necessarily periodic, the assertion follows. \end{proof}

Now we shall describe the relation between meromorphic $k$-differentials and meromorphic connections. A $k-$differential is a global meromorphic section of the line bundle $(T^*S)^{\otimes k}$. Meromorphic  $k$-differentials are studied by many authors (see for example \cite{BCG,Sch,ST,T}). It is not difficult to see that a $k$-differential $q$ is given locally as $q=q(z)dz^k$. Then it is possible to prove that there exists a singular flat metric $g$ locally given as $g^{\frac{1}{2}}=|q(z)|^{\frac{1}{k}}|dz|$. We say $g$ is \emph{adapted} to $q$. Similarly, a meromorphic connection $\nabla $ and $q$ are \emph{adapted} to each other if there exists a singular flat metric $g$ such that $g$ is adapted to $\nabla$ and $q$.

\begin{theorem}[{\cite[Theorem 1.2]{R1}}]\label{t6}
Let $\nabla$ be a Fuchsian meromorphic connection on a Riemann surface~$S$. If $\nabla$ has monodromy in $\mathbb{Z}_k:=\{\varepsilon\in S^1| \ \varepsilon^k=1\}$
then there is a meromorphic $k$-differential $q$ adapted to $\nabla$. Moreover, $q$ is  unique up to a non-zero constant multiple.

Conversely, if $q$ is a meromorphic $k$-differential on a Riemann surface $S$ then there exists a unique meromorphic connection $\nabla$ adapted to $q$. Moreover, $\nabla$ is Fuchsian  and it has monodromy in $\mathbb{Z}_k$.
\end{theorem}

\begin{remark}
    Note that if $\nabla$ is a Fuchsian meromorphic connection with  monodromy in $\mathbb{Z}_k$ then the residues of $\nabla$ are in $\frac{1}{k}\mathbb{Z}$, by Remark~\ref{perres}.
\end{remark} 

\section[Local behaviour around Fuchsian poles]{Local behaviour around Fuchsian poles}\label{localbehaviorfuchsianpolessec}
In this section we study the local behaviour of geodesics around non-resonant Fuchsian poles. In Subsection \ref{nonresonantfuchsian}, we show the existence a distinguished chart around a non-resonant Fuchsian pole of a meromorphic connection~$\nabla$. In Subsection \ref{geodesicsaroundnonresonatfuchsianpol} and Subsection \ref{s:noncriticalgeo}, we study the behaviour of geodesics of a meromorphic connection $\nabla$ around non-resonant and noncritical Fuchsian poles, respectively.  In Subsection \ref{s:singfmonchartUz}, we study the singular flat metric adapted to $\nabla$ in a neighbourhood of a Fuchsian pole. Finally, in Subsection \ref{sec5.3}, we study the $\omega$-limit set of a simple geodesic around a Fuchsian pole with real residue greater than $-1$.

\subsection{Non-resonant Fuchsian poles}\label{nonresonantfuchsian}
As we have recalled above, around any regular point $p\in S^o$  it is possible to find a $\nabla$-chart, that is a chart $(U,z)$ around $p$ such that the local representation of $\nabla$ is identically zero. Of course, this is not true around poles.   In the next lemma we show there exists a distinguished chart around non-resonant Fuchsian poles. 

\begin{lemma}\label{l8}
Let $\nabla$ be a  meromorphic connection on a Riemann surface $S$. Let $p_0\in S$ be a Fuchsian pole with residue $\rho=\operatorname{Res}_{p_0}\nabla\in\mathbb{C}$. Assume that $p_0$ is non-resonant, that is $\rho\ne-k-1$ for all $k\in\mathbb{N}^*$.
Then there exists a chart $(U,w)$ centered at $p_0$ such that the local representation of $\nabla$ on $U$ is of the form
\[
\eta=\frac{\rho}{w}dw. 
\]
\end{lemma}

\begin{proof}
Let $(V,z)$ be a simply connected chart centered at $p_0$ such that $p_0$ is the unique pole of $\nabla$ in~$V$. Since $p_0$ is a Fuchsian pole, the local representation of $\nabla$ on $V$ is of the form
\[
\eta_1=\left(\frac{\rho}{z}+f\right)dz,
\]
where $f\colon V\to\mathbb{C}$ is a holomorphic function. Let $F$ be a holomorphic primitive of $f$ with $F(0)=0$ and write the Taylor expansion of $e^F$ as
\begin{equation}\label{ch3.1}
  e^F=\sum_{j=0}^{\infty}c_jz^j,
\end{equation}
where $c_0=1.$ By shrinking $V$, if necessary, we can define a one-to-one holomorphic function by taking a branch of the following functions
\begin{equation}\label{ch3.2}
  K_\rho:=
\begin{cases}
    \left(\sum\limits_{j=0}^{\infty}\frac{c_j}{j+\rho+1}z^j\right)^{\frac{1}{\rho+1}},& \text{if } \rho\ne- 1,\\
    \exp \sum\limits_{j=1}^{\infty}\frac{c_j}{j}z^j,             &\text{if } \rho=-1.
\end{cases}
\end{equation}
such that
\[
K_\rho(0)=\begin{cases}
              \left(\frac{1}{\rho+1}\right)^{\frac{1}{\rho+1}}, & \mbox{if } \rho\ne-1 \\
              1, & \mbox{if } \rho=-1.
            \end{cases} 
\]
Note that $K_\rho(0)\ne 0$. Set
\[
w=zK_\rho\colon V\to \mathbb{C}.
\]
By shrinking again $V$, if necessary, we can assume that $w$ is injective; choose $r>0$ so that $w(V)$ contains $W=\{|w|<r\}$. Then $(U,w)$, with $U=w^{-1}(W)$, is a chart centered at $p_0$.  Let $\eta$ be the local representation of $\nabla$ on $U$; by the transformation rule we have
\[
\eta=\eta_1-\frac{d\xi}{\xi}\;,
\]
where $\xi=K_\rho +zK_\rho' $. Then we have
\[
\begin{aligned}
  \eta &=\left(\frac{\rho }{z}+f -\frac{2K_\rho' +zK_\rho'' }{K_\rho +zK_\rho' }\right)dz \\
   &=\frac{\rho K_\rho +\rho zK_\rho' -2zK_\rho' -z^2K_\rho'' +z (K_\rho +zK_\rho' )f}{z(K_\rho +zK_\rho' )}dz \\
   &=\frac{\rho}{w}dw+\frac{zf K_\rho^2 +(z^2f -2z-\rho z)K_\rho K_\rho' -z^2K_\rho K_\rho'' -\rho z^2{K_\rho'}^2}{zK_\rho (K_\rho +zK_\rho' )}dz.
\end{aligned}
\]
Hence it is enough to prove that
\[
zf  K_\rho^2  +(z^2f  -2z-\rho z)K_\rho  K_\rho'  -z^2K_\rho  K_\rho''  -\rho z^2 {K_\rho'}^2\equiv0.
\]
Since $K_\rho  \ne 0$ on $U$,  we can multiply both sides of the last equality by a branch of $K_\rho^{\rho-1}$  and divide by $z$. Set
\[
A:=f  K_\rho^{\rho+1}  +(zf  -2-\rho)K_\rho^\rho  K_\rho'  -zK_\rho^\rho  K_\rho''  -\rho zK_\rho^{\rho-1}{K_\rho'}^2;
\]
it is enough to prove $A\equiv0$. By using the following simple formulas
\begin{itemize}
  \item for $\rho\ne -1$
  \begin{enumerate}
    \item $\frac{(K_\rho^{\rho+1}  )'}{\rho+1}=K_\rho^\rho  K_\rho'  ;$
    \item $\frac{(K_\rho^{\rho+1}  )''}{\rho+1}=K_\rho^\rho  K_\rho''  +\rho K_\rho^{\rho-1}{K_\rho'}^2;$
  \end{enumerate}
  \item for $\rho= -1$
  \begin{enumerate}
    \item $(\log K_\rho)'=\frac{K_\rho'}{K_\rho};$
    \item $(\log K_\rho)''=\frac{K_\rho''  }{K_\rho  }-K_\rho^{-2}{K_\rho'}^2.$
  \end{enumerate}
\end{itemize}
we have
\[
A=\begin{cases}
  \frac{1}{\rho+1}\left((\rho+1)f  K_\rho^{\rho+1}  +zf  (K_\rho^{\rho+1}  )'-(\rho+2)(K_\rho^{\rho+1}  )'-z(K_\rho^{\rho+1}  )''\right), &\mbox{if } \rho\ne-1 \\
  f  +zf  (\log K_\rho  )'-(\log K_\rho  )'-z(\log K_\rho  )'', &\mbox{if } \rho=-1.
\end{cases}
\]
By using \eqref{ch3.1} and \eqref{ch3.2} we have
\[
\sum_{j=1}^{\infty}jc_jz^{j-1}=(e^F)'=fe^F=\begin{cases}
    (\rho+1)f  K_\rho^{\rho+1}  +zf  (K_\rho^{\rho+1}  )', & \mbox{if } \rho\ne-1,\\
    f  +zf  (\log K_\rho  )', & \mbox{if }\rho=-1.
  \end{cases}
\]
Furthemore,
\[
\sum_{j=1}^{\infty}jc_jz^{j-1}=\begin{cases}
    (\rho+2)(K_\rho^{\rho+1}  )'+z(K_\rho^{\rho+1}  )'', & \mbox{if } \rho\ne-1,\\
    (\log K_\rho  )'+z(\log K_\rho  )'', & \mbox{if }\rho=-1.
  \end{cases}
  \]
Consequently, $A\equiv0$ and hence
\[
\eta=\frac{\rho}{w}dw,
\]
as claimed.
\end{proof}

\begin{definition}
  Let $\nabla$ be a meromorphic connection on a Riemann surface $S$. Let $p$ be a non-resonant Fuchsian pole of $\nabla$. A simply connected chart $(U,z)$ centered at $p$ is said to be \emph{adapted} to $(\nabla,p)$ if
   \begin{itemize}
     \item the local representation of $\nabla$ on $U$ is $\eta:=\frac{\rho}{z}dz,$  where $\rho=\operatorname{Res}_p\nabla$; and
     \item $z(U)$ is a disc in $\mathbb{C}$ centered in $0$ of radius $r>0$.
   \end{itemize}
   We say $r$ is the \emph{radius} of $(U,z)$.
\end{definition}

Lemma \ref{l8} says that for any non-resonant Fuchsian pole $p$ of a meromorphic connection $\nabla$ on a Riemann surface $S$ there exists a chart adapted to $(\nabla,p)$. In general, for a resonant Fuchsian pole $p$ there may not exist a chart adapted to $(\nabla, p)$. However, the local study of geodesics around a resonant Fuchsian pole is already covered in \cite{R1}.

\subsection[Geodesics around non-resonant Fuchsian poles ]{Geodesics around non-resonant Fuchsian poles with real residues}\label{geodesicsaroundnonresonatfuchsianpol}
The next proposition describes the behaviour of geodesics of a meromorphic connection $\nabla$ around  non-resonant Fuchsian poles having real residues.

\begin{proposition}\label{prop2}
  Let $\nabla$ be a meromorphic connection on a Riemann surface $S$. Let $p$ be a non-resonant Fuchsian pole of $\nabla$ with residue $\rho\in\mathbb{R}$. Let $(U,z)$ be a chart adapted to $(\nabla,p)$ with radius $r>0$. Let $\sigma\colon [0,\varepsilon)\to U\setminus\{p\}$ be a geodesic of $\nabla$. Fix $0<\beta<\frac{\pi}{2}$. Let
\[
H_\rho:=\begin{cases}
              \{w\in \mathbb{C}\mid -\beta< \arg w<\pi +\beta, |w|>r^{\rho+1}\} & \mbox{if } \rho<-1; \\
             \{w\in \mathbb{C}\mid \operatorname{Im}\,w>0\} & \mbox{if } \rho=-1;\\
              \{w\in \mathbb{C}\mid -\beta< \arg w< \pi +\beta,\, 0<|w|<r^{\rho+1}\} & \mbox{if } \rho>-1.
            \end{cases}
\]
 For $\alpha\in[0,2\pi)$, let $\chi_\rho^\alpha\colon H_\rho\to \chi_\rho^\alpha(H_\rho)\subseteq z(U)$ be the holomorphic map given by
\[
\chi_\rho^\alpha(w)=\begin{cases}
     e^{i\alpha}w^{\frac{1}{\rho+1}},&  \mbox{if } \rho\ne-1;\\
    re^{iw} &\mbox{if } \rho=-1. 
 \end{cases}
\]
\begin{enumerate}
\item If $\rho\ne-1$ then a smooth curve $\sigma\colon [0,\varepsilon)\to U\setminus\{p\}$ is a geodesic for $\nabla$ if and only if there exists $\alpha\in[0,2\pi)$  such that
 $z(\sigma(t))=\chi^\alpha_\rho(at+b)$
 for some $a>0$ and $b\in \mathbb{C}$ with $\operatorname{Im}\,b\ge0$.
Furthermore, setting $\tilde{\sigma}=z\circ\sigma$  we have
\[
\begin{aligned}
  \alpha= &\begin{cases}
     \frac{1}{\rho+1}\arg\left(-\tilde{\sigma}(0)^\rho\tilde{\sigma}'(0)\right), & \mbox{if } \rho<0, \\
    \frac{1}{\rho+1} \arg\left(\tilde{\sigma}(0)^\rho\tilde{\sigma}'(0)\right), & \mbox{if } \rho\ge0.
   \end{cases}
\end{aligned}
\]
Moreover, $a= \left|(\rho+1)\tilde{\sigma}(0)^\rho\tilde{\sigma}'(0)\right| $ and $ b=  e^{-i\alpha(\rho+1)}\tilde{\sigma}(0)^{\rho+1}.$

\item If $\rho=-1$ then a smooth curve $\sigma\colon [0,\varepsilon)\to U\setminus\{p\}$ is a geodesic for $\nabla$ if and only if 
we have $z(\sigma(t))=\chi_{-1}^0(at+b)$
for some $a\in \mathbb{C}^*$ and $b\in \mathbb{C}$ with $\operatorname{Im}\,b\ge0$.
Moreover, setting $\tilde{\sigma}=z\circ\sigma$
we have $a=\frac{i\tilde{\sigma}'(0)}{\tilde{\sigma}(0)}$ and $e^{ib}= \frac{\tilde{\sigma}(0)}{r}.$
\end{enumerate}
\end{proposition}

\begin{figure}
\centering
\begin{subfigure}{.4\textwidth}
  \centering
  \includegraphics[width=.6\linewidth]{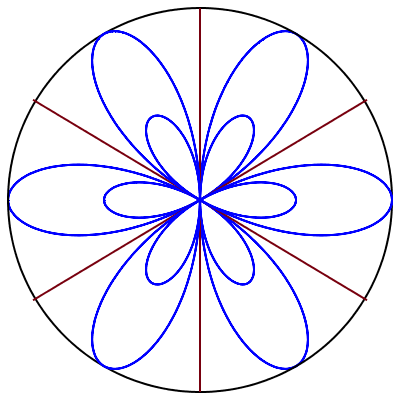}
  \caption{$\rho<-1$}
  \label{fig:rhokichik-1rasm}
\end{subfigure}
\begin{subfigure}{.4\textwidth}
  \centering
  \includegraphics[width=.6\linewidth]{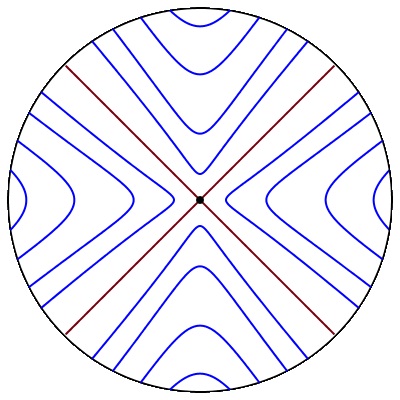}
  \caption{$\rho>-1$}
  \label{fig:rhokatta-1rasm}
\end{subfigure}
\caption{Geodesics of $\nabla$ on $z(U)$.}
\label{fig:-1gatengemaslari}
\end{figure}

\begin{proof}
A curve $\sigma\colon [0,\varepsilon)\to U\setminus\{p\}$ is a geodesic for $\nabla$ if and only if the representation  $\tilde{\sigma}=z\circ \sigma$ satisfies
  \begin{equation}\label{geodesictar}
  {\tilde{\sigma}}''+\frac{\rho}{\tilde{\sigma}} (\tilde{\sigma}' )^2=0.
  \end{equation}
  By solving \eqref{geodesictar} we get that $\tilde\sigma(t)=\chi^\alpha_\rho(at+b)$. The values of $\alpha$, $a$ and $b$ follow from $\tilde{\sigma}(0)=\chi^\alpha_\rho (b)$ and $\tilde{\sigma}'(0)=a(\chi^\alpha_\rho )'(b)$.
\end{proof}

Figures~\ref{fig:-1gatengemaslari} and~\ref{fig:-1gatengrasm} contain typical examples of geodesics.

\begin{remark}
Assume $\rho>-1$. In Proposition \ref{prop2} we can take
\[
\tilde{H}_\rho:= \{w\in \mathbb{C}\mid 0< \arg w< \pi,\, 0<|w|<r^{\rho+1}\}
\] 
instead of $H_\rho$. More precisely,  a smooth curve $\sigma\colon [0,\varepsilon)\to U\setminus\{p\}$ is a geodesic for $\nabla$ if and only if there exists $\alpha\in[0,2\pi)$  such that
\[
z(\sigma(t))=\chi^\alpha_\rho(l(t))
\]
for some horizontal Euclidean segment $l$ in $\tilde{H}_\rho$.
 \end{remark}

\begin{figure}
\centering
\begin{subfigure}{.4\textwidth}
  \centering
  \includegraphics[width=.6\linewidth]{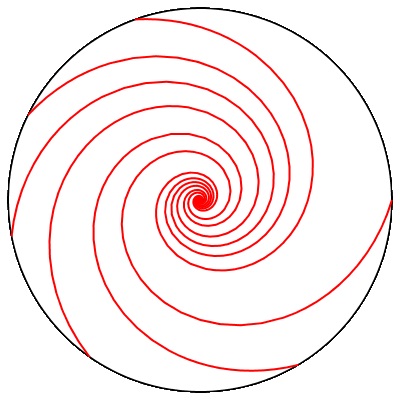}
  \caption{spirals}
  \label{fig:spirallarrasm}
\end{subfigure}
\begin{subfigure}{.4\textwidth}
  \centering
  \includegraphics[width=.6\linewidth]{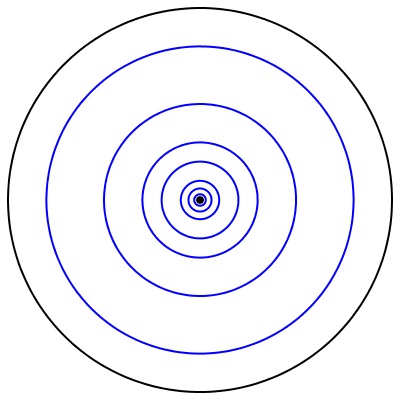}
  \caption{periodic geodesics}
  \label{fig:periodiclarrasm}
\end{subfigure}
\begin{subfigure}{.4\textwidth}
  \centering
  \includegraphics[width=.6\linewidth]{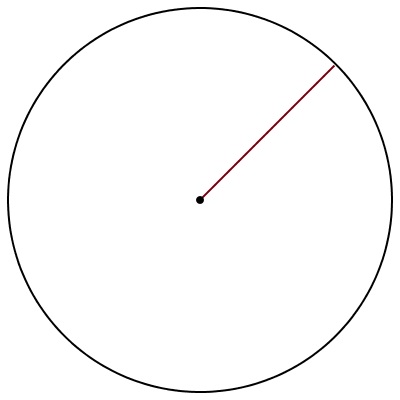}
\caption{a critical geodesic}
  \label{fig:criticalcha}
\end{subfigure}
\caption{$\rho=-1$}
\label{fig:-1gatengrasm}
\end{figure}

Actually,  Proposition \ref{prop2} is a special case of  \cite[Theorem 8.1]{AT1}. We shall need the explicit form of the parametrization of geodesics around non-resonant Fuchsian poles in the next sections. 

\subsection{Noncritical geodesics}\label{s:noncriticalgeo}

To study the $\omega$-limit set of a geodesic of a meromorphic connection $\nabla$ on a Riemann surface $S$ it is useful to know  the local behavior of geodesics around a pole $p$.

\begin{theorem}[{\cite{R1}}]\label{rhoge-1}
Let $\nabla$ be a meromorphic connection on a Riemann surface $S$. Set $S^o:=S\setminus\Sigma$ where $\Sigma$ is the set of poles for $\nabla$.  Let $\sigma\colon [0,\varepsilon)\to S^o$ be a maximal geodesic of $\nabla$ and $W$  its $\omega$-limit set. Let $p$ be a Fuchsian pole with $\operatorname{Re}\operatorname{Res}_{p}\nabla\le -1$. If $p\in W$ then $W=\{p\}$.
\end{theorem}

Thanks to Theorem \ref{rhoge-1}, we can see that studying the $\omega$-limit set of geodesics approaching poles with residues less than or equal to~$-1$ is trivial.  Therefore, in this section we study only Fuchsian poles with real $\operatorname{Res}_{p}\nabla> -1$. Since resonant Fuchsian poles have residue less than $-1$, we omit the word ``non-resonant"  when we speak of poles with real residue greater than $-1$.

\begin{definition}
  Let $\nabla$ be a meromorphic connection on a Riemann surface $S$. Let $p$ be a Fuchsian pole for $\nabla$ with real $\operatorname{Res}_p\nabla>-1$. Let $(U,z)$ be a chart adapted to $(\nabla,p)$. Let $\sigma\colon (\varepsilon_-,\varepsilon_+)\to U\setminus\{p\}$ be a maximal
  (maximal in $U$ in both forward and backward time) geodesic for $\nabla$. We say $\sigma$ is a \emph{noncritical geodesic} of $\nabla$ on $U$  if both rays of $\sigma$ tend to the boundary of $U$. Otherwise, we say $\sigma$ is \emph{critical}. See Figure~\ref{fig:criticalvanoncriticalbirga}.
\end{definition}

\begin{corollary}\label{noncriticalvacritical}
   Let $\nabla$ be a  meromorphic connection on a Riemann surface $S$. Let $p$ be a Fuchsian pole for $\nabla$ with real residue $\rho:=\operatorname{Res}_p\nabla>-1$. Let $(U,z)$ be a chart adapted to $(\nabla,p)$. Let $\sigma\colon (\varepsilon_-,\varepsilon_+)\to U\setminus\{p\}$ be a maximal
  (both in forward and in backward time) geodesic for $\nabla$. Then
  \begin{enumerate}
    \item $\sigma$ is a non-critical geodesic if and only if there exists $\alpha\in[0,2\pi)$ such that
\[
z(\sigma(t))=e^{i\alpha}(at+b)^{\frac{1}{\rho+1}}
\]
    for some $a>0$ and $b\in \mathbb{C}$ with $\operatorname{Im}\,b>0$. Moreover, $|a\varepsilon_{\pm}+b|=r^{\rho+1} $ and, up to changing the initial point, we can assume $\operatorname{Re}\, b=0$. 

    \item $\sigma$ is a critical geodesic if and only if there exists $\alpha\in[0,2\pi)$ such that
\[
z(\sigma(t)):=e^{i\alpha}(at+b)^{\frac{1}{\rho+1}}
\]
    for some $a>0$ and $b\in \mathbb{R}$ with $-\frac{b}{a}\in\{\varepsilon_-,\varepsilon_+\}$.  Actually, one can choose $b=\varepsilon_-=0$  up to changing the initial point.
  \end{enumerate} 
   Moreover, for each $q\in U\setminus\{p\}$ there exists a unique (up to reparametrization) critical geodesic issuing from~$q$.
\end{corollary}

\begin{figure}
\centering
\begin{subfigure}{.4\textwidth}
  \centering
  \includegraphics[width=.6\linewidth]{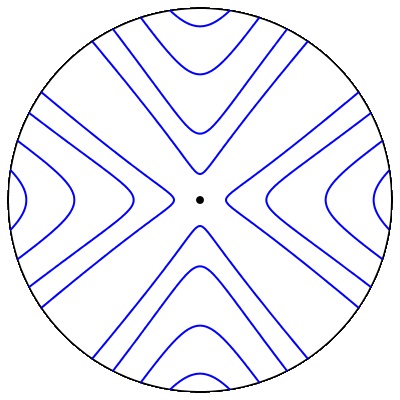}
  \caption{noncritical geodesics}
  \label{fig:noncritical}
\end{subfigure}
\begin{subfigure}{.4\textwidth}
  \centering
  \includegraphics[width=.6\linewidth]{critical}
  \caption{a critical geodesic}
  \label{fig:critical}
\end{subfigure}
\caption{Critical and noncritical geodesics on $(U,z)$.}
\label{fig:criticalvanoncriticalbirga}
\end{figure}

\begin{proof}
  By Proposition \ref{prop2} there exists $\alpha\in[0,2\pi)$ such that
\[
z(\sigma(t))=e^{i\alpha}(at+b)^{\frac{1}{\rho+1}}
\]
    for suitable $a>0$ and $b\in \mathbb{C}$ with $\operatorname{Im}\,b \ge 0$. Let $r$ be the radius of $(U,z)$. By definition $\sigma$ is noncritical if and only if $|z(\sigma(t))|\to r$ in both forward and backward time. Since $|z(\sigma(t))|=|at+b|^{\frac{1}{\rho+1}}$ we can see that $\sigma$ is noncritical if and only if $|at+b|\to r^{\rho+1}$ in both forward and backward time. Since $a>0$ it is not difficult to see that $|at+b|\to r^{\rho+1}$ both in forward and backward time  if and only if $\operatorname{Im}\,b \ne 0$.  Hence $\sigma$ is noncritical if and only if  $\operatorname{Im}\,b > 0$.   Consequently, $\sigma$ is critical if and only if $\operatorname{Im}\, b=0$. Moreover,  $|at+b|\to 0$ in a ray and hence $-\frac{b}{a}\in\{\varepsilon_-,\varepsilon_+\}$.
\end{proof}

\begin{corollary}
   Let $\nabla$ be a  meromorphic connection on a Riemann surface $S$. Let $p$ be a Fuchsian pole for $\nabla$ with real $\operatorname{Res}_p\nabla>-1$. Let $(U,z)$ be a chart adapted to $(\nabla,p)$. Let $\sigma\colon (\varepsilon_-,\varepsilon_+)\to U\setminus\{p\}$ be a critical geodesic. Then $\arg z(\sigma(t))=\mathrm{const}$. In particular $\mathrm{supp}(z(\sigma))$ is the support of a radius of $z(U)$.
\end{corollary}

\begin{proof}
By the previous corollary we can assume that $\varepsilon_-=0$ and that there exists $\alpha\in[0,2\pi)$ and $a>0$ such that
\[
z(\sigma(t))=e^{i\alpha}(at)^{\frac{1}{\rho+1}}.
\]
Hence $\arg z(\sigma(t))\equiv \alpha=\mathrm{const.}$
\end{proof}

\subsection{Singular flat metric on a chart  $(U,z)$ adapted to $(\nabla,p)$}\label{s:singfmonchartUz}
Let $\nabla$ be a meromorphic connection on a Riemann surface $S$. Let $p$ be a Fuchsian pole of $\nabla$ with real residue $\rho:=\operatorname{Res}_p\nabla>-1$. Let $(U,z)$ be a chart adapted to $(\nabla,p)$.  Note that the representation of $\nabla$ on $U$ is $\eta=\frac{\rho}{z}dz$. Let $g$ be a singular flat metric adapted to  $\nabla$ on $U$. Thanks to \cite[Lemma 3.3]{R1} it can be represented by
\begin{equation}\label{gniko`rinish}
  g^{\frac12}=|z|^\rho|dz|
\end{equation}
on $U$. We now collect a few properties of $\mathrm{dist}_g$, the distance induced by~$g$.

\begin{lemma}\label{2rrho+1}
Given $r>0$, let $g$ be a singular flat metric on $\Delta_r=\{z\in\mathbb{C}\mid |z|<r\}$ of the form~\eqref{gniko`rinish} with $\rho>-1$. Then 
\[
\mathrm{dist}_g(z_1,z_2)< \frac{2r^{\rho+1}}{\rho+1}
\]
for any $z_1$,~$z_2\in \Delta_r$.
\end{lemma}

\begin{proof}
  It is enough to show that there exists a curve $\sigma\colon [-1,1]\to \Delta_r$ connecting $z_1$ and $z_2$ with $g-$length bounded by $\frac{2r^{\rho+1}}{\rho+1}$. Set
\[
\sigma(t):=\begin{cases}
                 -z_1t, & \mbox{if } t\in[-1,0], \\
                 z_2t, & \mbox{if } t\in[0,1].
               \end{cases}
\]
Then $\sigma(-1)=z_1$ and $\sigma(1)=z_2$. It is not difficult to compute the $g-$length of $\sigma$:
\begin{equation}\label{sigmaguzunligi}
  \int_{\sigma}|z|^\rho|dz|=|z_1|^{\rho+1}\int_{-1}^{0}t^{\rho}dt+|z_2|^{\rho+1}\int_{0}^{1}t^{\rho}dt
  =\frac{|z_1|^{\rho+1}+|z_2|^{\rho+1}}{\rho+1}<\frac{2r^{\rho+1}}{\rho+1}
\end{equation}
and we are done.
\end{proof}

\begin{remark}
  If $\rho\le-1$ then $\mathrm{dist}_g$ is unbounded on~$\Delta_r$. In particular, a computation similar to \eqref{sigmaguzunligi} shows that an Euclidean segment joining $0$ and any point $z_2\in\Delta_r$ has infinite $g-$length.
\end{remark}

\begin{lemma}\label{lengthcriticalg}
   Let $\nabla$ be a meromorphic connection on a Riemann surface $S$. Let $p$ be a Fuchsian pole of~$\nabla$ with real residue greater than $-1$. Let $(U,z)$ be a chart adapted to $(\nabla,p)$. Let $g$ be a singular flat metric adapted to $\nabla$ on $U$. Then all maximal critical geodesics in $U$ have the same $g-$length.
\end{lemma}

\begin{proof}
By Corollary \ref{noncriticalvacritical} we can see that the support of the representation of a critical geodesic on $z(U)$ is an Euclidean segment joining $0$ with a boundary point of $z(U)=\Delta_r$. Then a computation similar to \eqref{sigmaguzunligi} shows that the length depends only the radius~$r$ of $z(U)$.
\end{proof}

In the next lemma we study self-intersections of noncritical geodesics  around a Fuchsian pole with residue in $(-1,-\frac{1}{2})$.

\begin{lemma}\label{fuchsianyaqin}
 Let $\nabla$ be a meromorphic connection on a Riemann surface $S$. Let $p$ be a Fuchsian pole of~$\nabla$ with real residue $\rho:=\operatorname{Res}_p\nabla$. Let $(U,z)$ be a chart adapted to $(\nabla,p)$ and $g$ is a singular flat metric adapted to $\nabla$ on $U$. 
 
 \begin{enumerate}
 \item If $\rho\in(-1,-\frac{1}{2})$ then there exists $\delta_0>0$ such that
any noncritical geodesic $\sigma:(\varepsilon_-,\varepsilon_+)\to U\setminus\{p\}$ of $\nabla$ on $U$ with $\mathrm{dist}_g(\mathrm{supp}(\sigma),p)<\delta_0$
intersects itself at least once. More precisely, any noncritical geodesic entering  $U^{\delta_0}=\{q\in U\mid\mathrm{dist}_g(q,p)<\delta_0\}$ intersects itself at least once before exiting $U^{\delta_0}$.

\item If $\rho>-1$ take $\alpha_1$,~$\alpha_2\in\mathbb{R}$ so that
\[
0<|\alpha_1-\alpha_2|<\frac{\pi}{\rho+1}.
\] 
If $L_1$ and $L_2$ are parametrizations of maximal horizontal Euclidean segments in 
\[
F_{\tau}=\{w\in H_\rho\mid 0<\operatorname{Im}\,w<\tau\},
\]
with $\tau$ small enough, then the geodesics
$\chi_\rho^{\alpha_1}\circ L_1$ and $\chi_\rho^{\alpha_2}\circ L_2$
have at least one common point.
\end{enumerate}
\end{lemma}

\begin{figure}[h]
    \centering
    \includegraphics[width=0.30\textwidth]{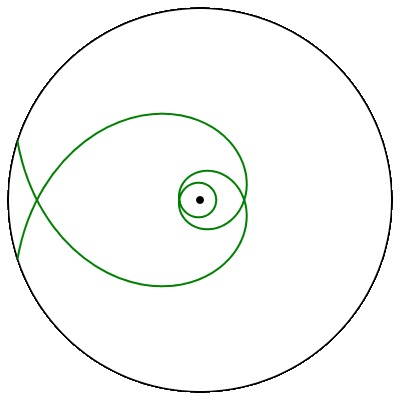}
    \caption{A self-intersecting noncritical geodesic on $z(U)$ (here $\rho=-0.9$).}
    \label{fig:kesadigan}
\end{figure}

\begin{proof}
 Let $r$ be the radius of $(U,z)$. For $0<\tau<r^{\frac{1}{\rho+1}}$, set
\[
l_\tau=\{w\in H_\rho\mid \operatorname{Im}\,w=\tau\}.
\]
 Let $L_\tau\colon(\varepsilon_-,\varepsilon_+)\to H_\rho$ be the parametrization of the horizontal Euclidean segment $l_\tau\subset H_\rho$ given by $L_\tau(t)= t+i\tau$.  Let
\[
\beta(\tau):=\sup_{w\in l_\tau}\arg w-\inf_{w\in l_\tau}\arg w=|\arg L_\tau(\varepsilon_-)-\arg L_\tau(\varepsilon_+)|.
\]
 Note that $\beta(\tau)$ is a decreasing function and that 
 \[
 \lim\limits_{\delta\searrow0}\beta(\tau)=\pi.
 \]
In case (1), since ${\rho+1}<\frac12$, we can choose  $\tau_0$ small enough so that $2\pi(\rho+1)<\beta(\tau)<\pi$ for all $0<\tau\le\tau_0$.

Fix $\alpha\in[0,2\pi)$. Then  
 \[
 \sigma_\tau(t):=\sigma_\tau^\alpha(t):=\chi^\alpha_\rho(L_\tau(t))=e^{i\alpha}(t+i\tau)^{\frac{1}{\rho+1}}
 \]
 is the representation of a noncritical geodesic of $\nabla$ on $U$. Set 
 \[
 \beta_\tau^1=\left|\lim\limits_{t\nearrow\varepsilon_+} \arg\sigma_\tau(t)-\lim\limits_{t\searrow\varepsilon_-} \arg\sigma_\tau(t)\right|.
 \]
It is easy to see that $\beta_\tau^1=\frac{1}{\rho+1}\beta_\tau$ and that $\arg\sigma_\tau(t)$ is a monotone function of $t$. Since $\beta_\tau^1>2\pi$ and we know that both rays of $\sigma_\tau$ tends to the boundary of $U$ and that $\arg\sigma_\tau(t)$ is monotone, we conclude that $\sigma_\tau$  intersects itself at least once. Then the image by $\chi^\alpha_\rho$ of any maximal horizontal interval in $F_{\tau_0}$  intersects itself. Set 
\[
\delta_0:=\mathrm{dist}_g(\mathrm{supp}(\sigma_{\tau_0}),p).
\]
Since multiplication by $e^{i\alpha}$ is an isometry for $g$ we deduce that any maximal noncritical geodesic $\sigma$ with  $\mathrm{dist}_g(\mathrm{supp}(\sigma),p)<\delta_0$ intersects itself, as claimed.

 In order to prove the second statement, without loss of generality we can assume that $\alpha_1>\alpha_2$. Set 
 \[
 \beta:=(\rho+1)(\alpha_1-\alpha_2)
 \]
and  $F_{\tau}^\beta=e^{i\beta}F_\tau$. Since $0<\beta<\pi$ we have $F_{\tau}^\beta\cap F_{\tau}\ne\emptyset.$ Let $L_\tau^\beta(t) := e^{i\beta}L_\tau(t).$ Since $0<\beta<\pi$, there exists $\tau_0$ such that  for any $0<\tau<\tau_0$ the segments  $l_\tau$ and $l_\tau^\beta$ have a common point. By
\[
\chi^{\alpha_1}_\rho\bigl(L_\tau(t)\bigr)=\chi^{\alpha_2}_\rho\bigl(L_\tau^\beta(t)\bigr)
\] 
and $l_\tau\cap l_\tau^\beta\ne\emptyset$, we can see that the supports of $\chi^{\alpha_1}_\rho\circ L_\tau$ and $\chi^{\alpha_2}_\rho\circ L_\tau$ have a common point.

 Let $L_1$ and $L_2$ be any maximal horizontal Euclidean segments in $F_\tau$. Since  $L_\tau$ and $L_\tau^\beta$ have a common point we can see that $L_1$ and $e^{i\beta}L_2$ have a common point. It follows as above that $\chi_\rho^{\alpha_1}\circ L_1$ and $\chi_\rho^{\alpha_2}\circ L_2$
  have at least one common point.
\end{proof}

\begin{definition}
Let $\nabla$ be a  meromorphic connection on a Riemann surface $S$. Let $p$ be a Fuchsian pole for $\nabla$ with $\rho:=\operatorname{Res}_p\nabla>-1$. Let $(U,z)$ be a chart adapted to $(\nabla,p)$. Let $\sigma\colon(\varepsilon_-,\varepsilon_+)\to U\setminus\{p\}$ be a maximal geodesic for $\nabla$ on $U$. We say $\sigma$ has \emph{$\alpha-$direction} for some $\alpha\in[0,2\pi)$ if $$z(\sigma)=\chi^\alpha_\rho(at+b)$$  for some $a\in\mathbb{R}^*$ and $b\in \mathbb{C}$. We shall write 
\[
\angle(\sigma):=\alpha.
\]
\end{definition}

Of course the notion of direction depends on $(U,z)$. In particular, if $\sigma$  has {$\alpha-$direction} for some $\alpha\in[0,2\pi)$ on $(U,z)$, then we can introduce a new chart $(U,z')$ with $z'=e^{-i\alpha}z$ so that $\sigma$  has {$0-$direction} in~$(U,z')$. 

\begin{corollary}\label{lengthkesish}
Let $\nabla$ be a  meromorphic connection on a Riemann surface $S$. Let $p$ be a Fuchsian pole for $\nabla$ with $\rho:=\operatorname{Res}_p\nabla>-1$. Let $(U,z)$ be a chart adapted to $(\nabla,p)$ and $g$ a singular flat metric adapted to $\nabla$ on $U$. Let $\beta_1$,~$\beta_2\in[0,2\pi)$ with $\beta_1<\beta_2$ be such that
\[
\min(\beta_2-\beta_1,2\pi+\beta_1-\beta_2)<\frac{\pi}{\rho+1}.
\]
Set 
\begin{equation}\label{Ialpha}
  I^\rho[\beta_1,\beta_2]:=\begin{cases}
                          [\beta_1,\beta_2], & \mbox{if } \beta_2-\beta_1<\frac{\pi}{\rho+1}, \\
                          [0,\beta_1]\cup[\beta_2,2\pi), & \mbox{otherwise}.
                        \end{cases}
\end{equation}
Then there exists a number $\delta_0>0$ depending only on the length of $I^\rho[\beta_1,\beta_2]$  such that any two noncritical geodesics $\sigma_1$ and $\sigma_2$ entering  
\[
U^{\delta_0}=\{q\in U\mid \mathrm{dist}_g(q,p)<\delta_0\}
\] 
with $\angle(\alpha_1),\angle(\alpha_2)\in I^\rho[\beta_1,\beta_2]$ and $\angle(\sigma_1)\ne\angle(\sigma_2)$ have a common point.

\end{corollary}
\begin{proof}
Without loss of generality we can assume that $\beta_2-\beta_1<\frac{\pi}{\rho+1}$. So  $I^\rho[\beta_1,\beta_2]=[\beta_1,\beta_2]$. Since $\angle(\sigma_1)\ne\angle(\sigma_2)$, the geodesics $\sigma_1$ and $\sigma_2$ are distinct. Since they are noncritical, the assertion follows from Lemma \ref{fuchsianyaqin}.(2).
\end{proof}

\subsection{The $\omega$-limit sets of  simple geodesics around Fuchsian poles}\label{sec5.3}
In this section we study the $\omega$-limit set of a simple geodesic around a Fuchsian pole with real residue greater than $-1$.

\begin{lemma}\label{cheklitakes}
  Let $\nabla$ be a meromorphic connection on a Riemann surface $S$. Set $S^0=S\setminus\Sigma$, where $\Sigma$ is the set of poles of $\nabla$. Let $p$ be a Fuchsian pole of $\nabla$ with $\rho:=\operatorname{Res}_p\nabla>-1$. Let $(U,z)$ be a chart adapted to $(\nabla,p)$ and $g$  a singular flat metric adapted to $\nabla$ on $U$.   Let $\sigma:[0,\varepsilon)\to S^0$ be a maximal simple geodesic for~$\nabla$. Then there exists a number $\delta_0>0$ such that $\sigma$ can enter $U^{\delta_0}=\{q\in U:\mathrm{dist}_g(q,p)<\delta_0\}$ only along finitely many directions.
\end{lemma}

\begin{proof}
 Pick a positive integer $n>1$ such that $\beta:=\frac{2\pi}{n}<\frac{\pi}{\rho+1}$. Choose $\beta_j\in(0,2\pi)$ for $j=1,...,n$ such that
\[
\bigcup_{j=1}^{n-1} [\beta_j,\beta_{j+1}]\cup I^\rho[\beta_n,\beta_{1}]=[0,2\pi)
\]
 and such that the lengths of $I_j:=[\beta_j,\beta_{j+1}]$ for $j=1,...,n-1$ and of $I_n:=I^\rho[\beta_n,\beta_{1}]=[0,\beta_1]\cup [\beta_n,2\pi)$ are all equal to $\beta$.
  Take $\delta_0=\delta_0(\beta)$ and $U^{\delta_0}=\{q\in U\mid \mathrm{dist}_g(q,p)<\delta_0\}$ given by Corollary \ref{lengthkesish}.
  Without loss of generality assume $\sigma(0)\notin U^{\delta_0}$.
  Let $q$ be a point of intersection between $\sigma$ and the boundary of $U^{\delta_0}$. If $\sigma$ in $q$ has the direction of the unique critical geodesic issuing from $q$ then $\sigma$ is that critical geodesic and, then, it does not exit from $U^{\delta_0}$ anymore. So in this case the assertion is proved.  
  
 Assume then that $\sigma$ does not enter $U^{\delta_0}$ in a critical direction.  If, by contradiction, $\sigma$ enters  $U^{\delta_0}$ along infinitely many directions, then there exists $I_j$ such that $\sigma$ enters to $U^{\delta_0}$ along two different directions $\alpha_1$,~$\alpha_2\in I_j$. Then Corollary \ref{lengthkesish} implies that $\sigma$ intersects itself, contradiction.
\end{proof}

In Theorem \ref{t3} we have seen the possible classifications of the $\omega$-limit sets of maximal simple geodesics of meromorphic connections on compact Riemann surfaces. In the next proposition we study the intersection of a neighbourhood  of a Fuchsian pole with residue greater than $-\frac{1}{2}$ with  the $\omega$-limit set of a maximal simple geodesic having non-empty interior.

\begin{proposition}\label{nonemptyinteriorandnonempty}
    Let $\nabla$ be a meromorphic connection on a compact Riemann surface $S$. Set $S^0=S\setminus\Sigma$, where $\Sigma$ is the set of poles of $\nabla$. Let $p$ be a Fuchsian pole of $\nabla$ with $\rho:=\operatorname{Res}_p\nabla\ge-\frac{1}{2}$. Let $\sigma\colon[0,\varepsilon)\to S^0$ be a maximal simple geodesic for $\nabla$ and $W$ its $\omega$-limit set. Assume that $W$ has nonempty interior and $p\in  W$.  Then there exists a chart $(U,z)$  adapted to $(\nabla,p)$ with radius $r_0>0$ such that
\[
z(U\cap W)=\bigcup\limits_{j=1}^m C_j\cup\{0\},
\] 
where
\[
C_j=\{\delta_1^j\le\arg z\le\delta^j_2, |z|<r_0\} 
\]
for some $\delta_1^j$,~$\delta_2^j\in \mathbb{R}$ so that $C_j\cap C_h=\emptyset$ for $j\ne h$. Moreover, 
\begin{equation}\label{tenglikcha}
          \delta^j_2-\delta^j_1=\frac{\pi k_j}{\rho+1}\in(0,2\pi]
\end{equation}
for some $k_j\in \mathbb{N}$.
 \end{proposition}

\begin{remark}
 Assume that $\rho\in(-1,-\frac{1}{2})$. If $p\in W$, then $\sigma$ must enter the neighbourhood $U^{\delta_0}$ given by Lemma~\ref{fuchsianyaqin}. Since $\sigma$ is simple, then it must enter $U^{\delta_0}$ as a critical geodesic, that is it must go directly to $p$. In particular, $W=\{p\}$.
\end{remark}

\begin{proof} 
Let $(U_0,z)$ be a chart adapted to $(\nabla,p)$  of radius $r_0$. 
By Theorem \ref{t3}, the boundary of $W$ is composed by a graph of saddle connections without spikes. If $p\in\mathring W$, up to shrinking $U_0$ if necessary, we can assume that $U_0\Subset W$.
If $p\in\partial W$, then, for $r_0$ small enough, $\partial W\cap U_0$ must consists of critical geodesics ending at~$p$.
Notice that if $\gamma\colon(\varepsilon_-,\varepsilon_+)\to U_0\setminus\{p\}$ is a critical geodesic for $\nabla$ then $\arg z\bigl(\gamma(t)\bigr)=\mathrm{const.}$ Consequently, there exists $m\ge1$  so that
\[
z(U_0\cap W)=\bigcup\limits_{j=1}^m C_j\cup\{0\}
\] 
where  $C_j\cap C_h=\emptyset$ for $j\ne h$ and
\[
C_j=\{\delta_1^j\le\arg z\le\delta^j_2, |z|<r_0\}
\] 
for some $\delta_1^j,\delta_2^j\in \mathbb{R}$ with 
\[
0<\delta^j_2-\delta^j_1\le2\pi.
\]
We just have to prove \eqref{tenglikcha}. 

Let $A_\sigma$ be the set of noncritical geodesics in $U_0$  obtained as intersection of the support of $\sigma$ with $U_0$; notice that $\sigma$ cannot intersect $U_0$ in a critical geodesic because otherwise 
its $\omega$-limit set would reduce to $p$. Since $\sigma$ is simple, any two geodesics in $A_\sigma$ are disjoint. Given an integer $n>0$, put
\begin{equation}\label{distsupp}
   A^n_\sigma=\left\{\gamma\in A_\sigma\biggm|\mathrm{dist}_g(\mathrm{supp}(\gamma),p)<\frac{1}{n}\right\},
\end{equation}
where $g$ is a singular flat metric adapted to $\nabla$ on $U_0$. Since $p\in W$, we have $A^n_\sigma\ne\emptyset$ for all $n>0$. Set
\[
B^n:=\{\alpha=\angle(\gamma)\mid\gamma\in A^n_\sigma\}\ne\emptyset.
\]
By Lemma \ref{cheklitakes}, there exists $n_0$ such that $B^n$ is a finite set for any $n>n_0$. Notice that $B^n\supseteq B^{n+1}$.  Then there exists $n_1$ such that for any $n\ge n_1$ we have $B^n=B^{n+1}=:B_0$. Furthermore, for any two different $\alpha_1$,~$\alpha_2\in B^n$ we have
\[
\min(|\alpha_2-\alpha_1|,2\pi-|\alpha_2-\alpha_1|)\ge\frac{\pi}{\rho+1}.
\]
Otherwise,  Corollary \ref{lengthkesish} would imply that $\sigma$ is not simple, impossible.

Choose $r<\min\{r_0,\frac{1}{n_1}\}$ and $U\subset U_0$ such that $z(U)$ is the disk of radius $r$. Take any $\alpha_j\in B_0$.
 Then $z(U_{\alpha_j}):=\chi^{\alpha_j}_\rho(\tilde{H}_\rho)$ is a sector in $z(U)$ with angle $\frac{\pi}{\rho+1}$, where  $\tilde{H}_\rho:= \{w\in \mathbb{C}\mid 0\le \arg w\le \pi,\, 0<|w|<r^{\rho+1}\}$. Since $\rho\ge-\frac12$ we have that $\chi^{\alpha_j}_\rho$ is one-to-one.  Notice that if $\sigma$ enters  $U$  in the sector $U_{\alpha_j}$ it leaves $U$ again by the sector $U_{\alpha_j}$. As we have shown in the first part of the proof, the interior of $z(W\cap U_{\alpha_j})$ is a sector or it is empty. Assume, by contradiction, that $z(W\cap U_{\alpha_j})$ is a sector with angle less than $\frac{\pi}{\rho+1}$; then
\[
S_{\alpha_j}:=(\chi^{\alpha_j}_\rho)^{-1}(z(W\cap U_{\alpha_j}))\subset \tilde{H}_\rho
\]
 is a sector with angle less than $\pi$.  Since $\alpha_j\in B_0$, the set $\{\gamma\in A_\sigma^n\mid \angle(\gamma)=\alpha_j\}$ is infinite for all $n$.
Then there exists a sequence of maximal horizontal Euclidean segments $\{l_k\}$ in  $\tilde{H}_\rho$ with
\[
\mathrm{dist}(l_k,0)<\frac1k
\]
and such that there is $\sigma_k\in A^\sigma$ so that $\sigma_k(t)=\chi^{\alpha_j}_\rho(l_k(t))$, where dist is the Euclidean distance in $\tilde{H}_\rho$. Then there exists a number $k_0$ such that $l_k$ leaves the sector $S_{\alpha_j}$ for any $k\ge k_0$ and such that $\mathrm{supp}(\sigma_{k_0})\cap (W\cap U_{\alpha_j})\ne\emptyset$; hence 
\begin{equation}\label{eq:suppsigmanew}
    \mathrm{supp}(\sigma_{k_0})\not\subset W.
\end{equation}
 Since $\sigma$ enters $U_{\alpha_j}$ only in $\alpha_j$-direction, there exists a sequence of maximal horizontal Euclidean segments $\{L_h\}$ on $\tilde{H}_\rho$ such that $\chi_\rho^{\alpha_j}\circ L_h\in A_\sigma$ and that $\{L_h\}$ accumulates  $l_{k_0}\cap S_{\alpha_j}$. Since $L_h$ and  $l_{k_0}$ are horizontal Euclidean segments in  $\tilde{H}_\rho$, if $\{L_h\}$ accumulates a subset of $l_{k_0}$ then it accumulates the whole $l_{k_0}$. Hence $\mathrm{supp}(\sigma_{k_0})\subset W$, against \eqref{eq:suppsigmanew}.  
 
Thus, $z(W\cap U_{\alpha_j})$ is a sector with angle greater or equal to $\frac{\pi}{\rho+1}$. Since $z(U_{\alpha_j})$ is a sector with angle $\frac{\pi}{\rho+1}$, we have $U_{\alpha_j}\subset W$. 
So we have proved that $W\cap U = 
\bigcup_{\alpha\in B_0} U_\alpha$.   Since 
\[
\min(|\alpha_j-\alpha_i|,2\pi-|\alpha_j-\alpha_i|)\ge\frac{\pi}{\rho+1},
\] 
for any different $\alpha_j,\alpha_i\in B_0$, we can see that $U_{\alpha_i}\cap U_{\alpha_j}$ has empty interior.  Finally, the angle of the sector $z(U_{\alpha_j})$ is  $\frac{\pi}{\rho+1}$ and we are done.
\end{proof}

\begin{remark}\label{r:rphoin12N}
In Proposition \ref{nonemptyinteriorandnonempty}, if $p\in\mathring W$ then $m=1$ and $2\pi=\delta^1_2-\delta^1_1=\frac{\pi k}{\rho+1}=2\pi$ for a suitable $k\in\mathbb{N}$. Hence, we have $\rho\in \frac{1}{2}(\mathbb{N}\cup \{-1\})$.    
\end{remark}

A consequence of the previous proposition is the following corollary.

\begin{corollary}
Let $\nabla$ be a  meromorphic connection on a compact Riemann surface $S$. Set $S^o=S\setminus\Sigma$, where $\Sigma$ is the set of poles of $\nabla$. Let $p$ be a  Fuchsian critical point for $\nabla$  with a residue $\rho:=\operatorname{Res}_p\nabla\in \mathbb{R}\setminus\frac{1}{2}(\mathbb{N}\cup \{-1\})$.  Let $\sigma:[0,\varepsilon)\to S^o$ be a maximal geodesic and $W$ its $\omega$-limit set. If $p$ is an interior point of $W$ then $\sigma$ intersect itself infinitely many times.
\end{corollary}

\begin{proof}
Suppose first $\rho\le-1$. Then, by Theorem \ref{rhoge-1}, if $p\in W$ then $W=\{p\}$, contradiction. Hence $\rho>-1$.

Let $(U,z)$ be a chart adapted to $(\nabla,p)$ with $U\subset W$. Let $g$ be a singular flat metric adapted to $\nabla$ on~$U$. Let first assume $-1<\rho<-\frac{1}{2}$. Then, by Lemma \ref{fuchsianyaqin}, there exists a positive number $\delta$ such that
   any noncritical geodesic entering $U^\delta:=\{q\in U:\mathrm{dist}_g(p,q)<\delta\}$ intersects itself. Since $p\in W$, it is easy to see that $\sigma$ must enter $U^\delta$ infinitely many times.  Hence $\sigma$ intersects itself  in a  neighbourhood of $p$ infinitely many times.

Let assume now $\rho\ge -\frac12$. Assume, by contradiction, that $\sigma$ intersect itself finitely many times. Then up to changing the starting point we can assume that $\sigma$ does not intersect itself. Since $p\in\mathring W$, Remark \ref{r:rphoin12N} implies that
\[ \rho\in \frac{1}{2}(\mathbb{N}\cup \{-1\}),\] 
and this  contradicts the assumption $\rho\in \mathbb{R}\setminus\frac{1}{2}(\mathbb{N}\cup \{-1\})$. Hence, $\sigma$ must intersect itself infinitely many times.
\end{proof}

 Another possible $\omega$-limit set of a maximal simple geodesic for a meromorphic connection $\nabla$ on a compact Riemann surface $S$ is a boundary graph of saddle connections. In the next proposition we study the intersection of such a boundary graph of saddle connections with a neighbourhood  of a Fuchsian pole with real residue greater than $-1$.
 
\begin{proposition}\label{orasi2pitaqsim}
   Let $\nabla$ be a  meromorphic connection on a compact Riemann surface $S$. Set $S^o=S\setminus\Sigma$, where $\Sigma$ is the set of poles of $\nabla$. Let $p$ be a Fuchsian pole for $\nabla$ with $\rho:=\operatorname{Res}_p\nabla>-1$. Let $\sigma\colon [0,\varepsilon)\to S^0$ be a maximal simple geodesic for $\nabla$ and $W$ its $\omega$-limit set.  Assume $W$ is a boundary graph of saddle connections and $p\in \partial W$. Then there exists a chart  $(U,z)$ adapted to $(\nabla,p)$ such that
\[
z(U\cap W)=\bigcup\limits_{j=1}^m L_j\cup\{0\}
\]
where
\[
L_j=\{w\in z(U)\mid \arg w=\delta_j\}
\] 
for some $\delta_j\in \mathbb{R}$. Moreover, for any $L_j$ there exists $L_k$ such that
\[
\delta_j-\delta_k=\frac{\pi }{\rho+1} \mod 2\pi.
\]
\end{proposition}

\begin{proof}
Let $(U,z)$ be a chart  adapted to $(\nabla,p)$ and $r$ its radius. Note that if $\gamma\colon(\varepsilon_-,\varepsilon_+)\to U\setminus\{p\}$ is a critical geodesic for $\nabla$ then $\arg z(\gamma(t))=\mathrm{const.}$ Consequently, we can choose $r$ small enough so that there exists $m\ge1$  such that
\[
z(U\cap W)=\bigcup\limits_{j=1}^m L_j\cup\{0\}
\] 
where $L_j$ is as in the statement. 
Since $\sigma$ accumulates $L_j$,  there exists a sequence of noncritical geodesics   $\gamma_n\colon(\varepsilon^n_-,\varepsilon^n_+)\to U\setminus\{p\}$ accumulating $L_j$ such that  $\mathrm{supp}(\gamma_n)\subset\mathrm{supp}(\sigma)$ and
\begin{equation}\label{distsupp2}
\mathrm{dist}_g(\mathrm{supp}(\gamma_n),p)<\frac{1}{n}.
\end{equation}
Then there exist $a_n>0$ and $b_n\in \mathbb{C}$ with $\operatorname{Im}\,b_n>0$ such that
\[
z(\gamma_n(t))=e^{i\alpha_n}(a_nt+b_n)^{\frac{1}{\rho+1}}.
\]
Up to reparametrization we can assume $a_n=1$ and $\operatorname{Re}\,b_n=0$. Furthermore, \eqref{distsupp2} implies that $b_n\to 0$. Then $\varepsilon^n_-=-\varepsilon^n_+$.   
Let $\{n_k\}_{k=1}^\infty$ be a sequence such that the sequence $\{\alpha_{n_k}\}$ converges to $\alpha_0\in[0,2\pi)$.  Then $\gamma_{n_k}$ accumulates geodesics $\gamma^{\pm}\colon(0, \varepsilon_+)\to U\setminus\{p\}$ such that
\[
z(\gamma^{\pm}(t))=e^{i\alpha_0}(\pm t)^{\frac{1}{\rho+1}}.
\]
Then setting $\delta^{\pm}=\arg(z(\gamma^{\pm}(t)))$ we have $\delta^+-\delta^-=\frac{\pi}{\rho+1}\mod 2\pi$.
Note that $\mathrm{supp}(\gamma^{\pm})\subset W\cap U$. Moreover,
by construction, $\mathrm{supp}(\gamma^{+})\cup \mathrm{supp}(\gamma^{-})=L_j\cup L_k$
for some $L_k$ and thus we are done.
\end{proof}

\begin{corollary}\label{boundaryburchak}
Let $\nabla$ be a Fuchsian meromorphic connection  with real periods on a compact Riemann surface $S$. Set $S^o=S\setminus\Sigma$, where $\Sigma$ is the set of poles of $\nabla$. Let $\sigma:[0,\varepsilon)\to S^o$ be a maximal simple geodesic for $\nabla$. Assume the $\omega$-limit set of $\sigma$ is a boundary graph of  saddle connections $\Gamma=\partial V$ for some connected open set $V\subset S$. Then the vertices of $\Gamma$ are Fuchsian poles with residues greater than or equal to~$-\frac{1}{2}$; moreover, the interior (with respect to $V$) angle on a vertex $p_j$ of $\Gamma$ is $\frac{\pi}{\operatorname{Res}_{p_j}\nabla+1}$.
 \end{corollary}
  
\begin{proof}
  By Theorem \ref{rhoge-1} we know that $\operatorname{Res}_{p_j}\nabla>-1$. Since $\sigma$ does not intersect itself, by Lemma \ref{fuchsianyaqin} we have $\operatorname{Res}_{p_j}\nabla\ge-\frac{1}{2}$. Let $(U_j,z_j)$ be a chart adapted to $(\nabla,p_j)$ with radius $r$. Since $U_j\cap \Gamma$ is composed only by critical geodesics, when $r$ is small enough the first statement of Proposition \ref{nonemptyinteriorandnonempty}  implies that
\[
z_j(U\cap V)=\bigcup_{i=1}^{m_j} C_i\cup\{0\},
\]
where 
\[
C_i=\{w\in z(U)\mid\delta^1_i<\arg w<\delta^2_i\}
\] 
for some $\delta^1_i$,~$\delta^2_i\in\mathbb{R}.$  Note that $\sigma$ enters each $z^{-1}(C_i)$ infinitely many times. Arguing as in the proof of Proposition \ref{orasi2pitaqsim} we then obtain that 
\[
\delta^2_i-\delta^1_i=\frac{\pi}{\operatorname{Res}_{p_j}\nabla+1} \mod 2\pi,
\]
as claimed.
\end{proof}

\section{Generalized Teichm{\"u}ller formula}\label{ATTsection}

Let us recall a formula that was first proved for $\mathbb{P}^1(\mathbb{C})$ in  \cite{AT1} and later for any compact Riemann surface in \cite{AB}.

\begin{theorem}[{\cite[Theorem 4.1]{AT1}, \cite[Theorem 3.1]{AB}}]\label{fabrizioequality}
 Let $\nabla$ be a meromorphic connection on a compact Riemann surface $S$, with poles $\{p_1,...,p_r\}$, and set $S^o=S\setminus\{p_1,...,p_r\}$. Let $P$ be a part of $S$ whose boundary  multicurve $\gamma\subset S^o$ is regular and it has $m_f\ge1$ free components, positively oriented with respect to $P$. Let $z_1,...,z_s$ denote the vertices of the free components of $\gamma$ and $\varepsilon_j\in(-\pi,\pi)$ the external angle at $z_j$. Suppose that $P$ contains the poles $\{p_1,...,p_g\}$ and denote by $g_{\hat{P}}$ the genus of the filling $\widetilde{P}$ of $P$. Then
\begin{equation}\label{geodesicequality2}
  \sum_{j=1}^{s}\varepsilon_j=2\pi\left(2-m_f-2g_{\widetilde{P}}+\sum_{j=1}^{g}\operatorname{Re}\,\mathrm{ Res}_{p_j}(\nabla)\right).
  \end{equation}
\end{theorem}

In this section we prove some generalizations of this formula. Let us first study the existence of broken geodesics connecting two regular points in a neighbourhood of a pole with real residue greater than $-1$.

\begin{lemma}
 Let $\nabla$ be a meromorphic connection on a Riemann surface $S$. Let $p$ be a Fuchsian pole of~$\nabla$ with $\rho:=\operatorname{Res}_p\nabla>-1$. Let $(U,z)$ be a chart adapted to $(\nabla,p)$. For $\alpha\in[0,2\pi)$, let $U_\alpha\subset U$ be  such that
\[
V_\alpha:=z(U_\alpha):=\left\{w\in z(U)\biggm| \alpha<\arg w<\alpha+\frac{\pi}{\rho+1}\right\}.
\] 
Then any two points in $U_\alpha$ can be connected by a simple geodesic not leaving $U_\alpha$.
\end{lemma}

\begin{proof}
For $z\in V_\alpha$ put  $J(z)=(\chi_\rho^\alpha)^{-1}(z)=e^{-(\rho+1)\alpha}z^{\rho+1}$. It is not difficult to see that $J$ is a local isometry of $\nabla$ on $U_\alpha$ and that $J(V_\alpha)$ is an upper half disc. Since it is convex, any two points in $J(V_\alpha)$ can be joined a Euclidean segment not leaving $J(V_\alpha)$ and we are done.
\end{proof}

\begin{corollary}\label{graphichidanchiq}
Let $\nabla$ be a meromorphic connection on a Riemann surface $S$. Let $p$ be a Fuchsian pole of $\nabla$ with $\operatorname{Res}_p\nabla>-1$. Let  $(U,z)$ be a chart adapted to $(\nabla,p)$. Let $U_1\subset U$ be such that 
\[
z(U_1):=\{w\in z(U)\mid\alpha_1<\arg w<\alpha_2\}
\] 
for some $\alpha_1$, $\alpha_2\in\mathbb{R}$. Then any two points in $U_1$ can be joined by a simple broken geodesic not leaving $U_1$.
\end{corollary}

\begin{proof}
It follows from the previous lemma.
\end{proof}

\subsection{Generalized Teichm{\"u}ller formula}\label{abateteichmullertovenateng}
 In this subsection we introduce a generalized Teichm{\"u}ller formula, which is an extension of \eqref{geodesicequality2}.
 
\begin{definition}
  Let $\nabla$ be a meromorphic connection on a Riemann surface $S$. Let $R_0$ be a geodesic polygon and $p$ a vertex of $R_0$. Let $\varepsilon\in(-\pi,\pi)$ be the external angle of $R_0$ in $p$. We say that $v=\pi-\varepsilon$ is the \emph{internal angle} of $R_0$ in $p$.
\end{definition}

Let us prove the main lemma of this section.

\begin{lemma}\label{equalityadaptedda}
 Let $\nabla$ be a meromorphic connection on a Riemann surface $S$. Let $p_0$ be a Fuchsian pole for $\nabla$ with $\rho:=\operatorname{Res}_{p_0}\nabla>-1$. Let $(U,z)$ be a chart adapted to $(\nabla,p_0)$.  Let $R_0\subset U$ be an $s-$sided Fuchsian geodesic polygon with vertices $p_0,...,p_{s-1}$ such that $p_j\ne p_0$ for $j=1,...,s-1$. Let $v_j\in[0,2\pi)$ be the internal angle in $p_j$. Then
\begin{equation}\label{geodesicequalityfuchs}
  \sum_{j=1}^{s-1}\left(\pi-v_j\right)=\pi+v_0({\rho+1})\;.
\end{equation}
\end{lemma}

\begin{proof}
Let denote by $\sigma_j:[0,l_j]\to U$ the geodesics composing $\partial R_0$.
Since $R_0\subset U$ is a simply connected domain and $p_0\in \partial R_0$, there exists a non-self-intersecting smooth curve $\sigma$ in $\overline{U}\setminus R_0$ connecting $p_0$ with a point in $\partial U$. Set $U_1=U\setminus\mathrm{supp}(\sigma)$. Let
\[
\eta=\frac{\rho}{z}dz
\]
 be the local representation of $\nabla$ on $U$. Note that $U_1$ is a simply connected domain and $\eta$ is holomorphic on~$U_1$. Then $J(z)=\frac{z^{\rho+1}}{\rho+1}$ is  a local isometry of $\nabla$ on $U_1$.  Let $\gamma:[0,\varepsilon)\to U_1$ be a geodesic of $\nabla$. Then there exists $a\in\mathbb{C}$ and $b\in\mathbb{C}$ such that $$J(z(\gamma(t)))=at+b.$$
 Hence we have
\[
\bigl(z(\gamma)\bigr)^\rho dz(\gamma')=d(J\circ z)_\gamma(\gamma')\equiv a.
\]
Thus,
\[
0=d(\arg d(J\circ z)(\gamma'))= d(\arg (z(\gamma))^\rho)+d(\arg dz(\gamma')).
\]
Since $\partial R_0\setminus\{p_0\}$ is composed by geodesics we then have
\[
\int_{\partial R_0\setminus\{p_0\}}d(\arg z^{\rho})=-\int_{\partial R_0\setminus\{p_0\}}d(\arg dz).
\]
It is easy to see that
\[
\begin{aligned}
  \int\limits_{\partial R_0\setminus\{p_0\}}d(\arg z^{\rho})= & \int\limits_{\sigma_1\mid_{(0,l_1]}} d(\arg z^{\rho})+\sum_{j=2}^{s-1}\int\limits_{\sigma_j} d(\arg z^{\rho})+\int\limits_{\sigma_s\mid_{[0,l_s)}} d(\arg z^{\rho}) \\
  = & \arg z^\rho(\sigma_1(l_1))-\lim_{t\searrow 0}\arg z^\rho(\sigma_1(t))\\
  & +\sum_{j=2}^{s-1}\Bigl(\arg z^\rho(\sigma_j(l_j))-\arg z^\rho( \sigma_j(0))\Bigr)\\
  & + \lim_{t\nearrow l_s}\arg z^\rho(\sigma_s(t))-\arg z^\rho( \sigma_s(0))\\
  =& :{\sum}_1
\end{aligned}
\]
Since $\sigma_j(l_j)=\sigma_{j+1}(0)=p_{j}$  for $j=1,...,s-1,$ and $\sigma_s(l_s)=\sigma_1(0)=p_0$, we have
\[
{\sum}_1=\rho\left(\lim_{t\nearrow l_s}\arg z(\sigma_s(t))-\lim_{t\searrow 0}\arg z(\sigma_1(t))\right).
\]
Since $\sigma_1$ and $\sigma_s$ are critical geodesics of $\nabla$ on $U$, we also have $\arg z(\sigma_1(t))=\alpha_1$  and $\arg z(\sigma_s(t))=\alpha_s$ for some $\alpha_1$,~$\alpha_s\in\mathbb{R}$ with $\alpha_s-\alpha_1=v_0$. Hence
\[
\lim_{t\nearrow l_s}\arg z(\sigma_s(t))-\lim_{t\searrow 0}\arg z(\sigma_1(t))=v_0.
\]
Consequently,
\[
\int\limits_{\partial R_0\setminus\{p_0\}}d(\arg z^{\rho})=\rho v_0.
\]
On the other hand,
\[
\begin{aligned}
  \int\limits_{\partial R_0\setminus\{p_0\}}d(\arg dz) =& \int\limits_{\sigma_1\mid_{(0,l_1]}}d(\arg dz)+\sum_{j=2}^{s-1}\int\limits_{\sigma_j}d(\arg dz)+\int\limits_{\sigma_s\mid_{[0,l_s)}}d(\arg dz) \\
  =&\sum_{j=1}^s\int_{0}^{l_j}d(\arg dz(\sigma_j'(t)))\\
   =&\sum_{j=1}^s(\arg dz(\sigma_j')(l_j))-\arg dz(\sigma_j')(0)) \\
   =&\sum_{j=1}^{s-1}(\arg dz(\sigma_j')(l_j)-\arg dz(\sigma_{j+1}')(0))  \\
   &+\arg dz(\sigma_s')(l_s)-\arg dz(\sigma_1')(0)\\
   =&:{\sum}_2.
\end{aligned}
\]
Since
\[
\arg dz(\sigma_j')(l_j)-\arg dz(\sigma_{j+1}')(0)=v_j-\pi
\]
and
\[
\arg dz(\sigma_s')(l_s)-\arg dz(\sigma_1')(0)=v_0+\pi\;,
\]
we have
\[
{\sum}_2=\sum_{j=1}^{s-1}(v_j-\pi) +\pi+v_0.
\]
 Consequently,
\[
   -\rho v_0  =\sum_{j=1}^{s-1}(v_j-\pi) +\pi+v_0,
\]
 that is,
\[
      \sum_{j=1}^{s-1}(\pi-v_j)  =\pi+v_0({\rho+1}),
\]
as claimed.
\end{proof}

\begin{theorem}\label{t:attf}
  Let $\nabla$ be a meromorphic connection on a compact Riemann surface $S$ with poles $\{p_1,...,p_r\}$ and set $S^o=S\setminus\{p_1,...,p_r\}$. Let $P$ be a part of $S$ whose boundary is a Fuchsian multicurve $\gamma\subset S$ with $m_f\ge1$ free components, positively oriented with respect to $P$. Let $q_1,...,q_s$ denote the vertices of the free components of $\gamma$ and $v_j\in(0,2\pi)$ the internal angle at $q_j$; moreover, assume that $\rho_j:=\operatorname{Res}_{q_j}\nabla>-1$ for $j=1,\ldots,s$. Suppose that $P$ contains the poles $\{p_1,...,p_g\}$ and denote by $g_{\widetilde{P}}$ the genus of the filling $\widetilde{P}$ of~$P$. Then
\begin{equation}\label{genTeich}
  \sum_{j=1}^{s}\left(\pi-(\rho_j+1)v_j\right)=2\pi\left(2-m_f-2g_{\widetilde{P}}+\sum_{j=1}^{g}\operatorname{Re}\,\mathrm{ Res}_{p_j}(\nabla)\right).
\end{equation}
\end{theorem}

\begin{proof}
Denote by $\sigma_j\colon [0,l_j]\to S$ the smooth geodesics composing $\partial P$.  When $q_j$ is a pole, let $(U_j,z_j)$ be a chart adapted to  $(\nabla,q_j)$. By Corollary \ref{graphichidanchiq} and recalling the shape of critical geodesics,  there exists a broken geodesic $\Lambda_j$ connecting  $\sigma_j$ and $\sigma_{j+1}$ not leaving $U_j\cap P$ and not intersecting $\partial P$ outside the extremes. Let $z^1_j\in \sigma_j$ and $z^2_j\in\sigma_{j+1}$ the extremes of $\Lambda_j$.
Let $P_j$ be the simply connected component of $P\setminus \Lambda_j$ containing $q_j$ in its boundary. When $q_j$ is not a pole, put $P_j=\emptyset$. 
 Set
\[
\hat{P}=P\setminus\bigcup\limits_{j=1}^s\overline{P}_j.
\]
By construction $\hat{P}$ is a regular part such that the filling of $P$ and the filling of $\hat{P}$ have the same genus. For $k=1$,~$2$, denote by $\alpha^k_j$ the interior angle of $\hat{P}$ at $z^k_j$;  we put $\alpha^k_j=\pi$ when $q_j$ is not a pole. Denote other interior angles by $\beta_m$. By Theorem \ref{fabrizioequality}, we have
\begin{equation}\label{ch3s5}
  \sum_{j=1}^{s}\sum_{k=1}^2(\pi-\alpha^k_j)+\sum_{m}(\pi-\beta_m)=2\pi\left(2-m_f-2g_{\widetilde{P}}+\sum_{j=1}^{g}\operatorname{Re}\,\mathrm{ Res}_{p_j}(\nabla)\right).
\end{equation}
By using Lemma \ref{equalityadaptedda} for $P_j$ when $q_j$ is a pole we have
\[
\begin{aligned}
   \pi+(\rho_j+1)v_j=& \sum_{k=1}^{2}(\pi-(\pi-\alpha_j^k))+\sum_{m_j}(\pi-(2\pi-\beta_{m_j}))\\
  =& \sum_{k=1}^{2}\alpha_j^k+\sum_{m_j}(\beta_{m_j}-\pi),
\end{aligned}
\]
where $\beta_{m_j}$ is an interior angle of $\hat{P}$ adjacent to $P_j$. By summing up the last equality on all poles we have
\begin{equation}\label{tenglik12}
  \sum_{n=1}^{s'}\sum_{k=1}^2\alpha_{j_n}^k+\sum_{m_{j_n}}(\beta_{m_{j_n}}-\pi)=\sum_{n=1}^{s'}(\pi+(\rho_{j_n}+1)v_{j_n}),
\end{equation}
where $j_n\in\{1,...,s\}$ is such that $q_{j_n}$ is a pole and $s'$ is the number of poles in $\{q_1,...,q_s\}$. Since $\alpha^k_j=\pi$ when $q_j$ is not a pole, comparing \eqref{ch3s5} and \eqref{tenglik12} we have
\[
2\pi s-\sum_{j=1}^{s}(\pi+(\rho_j+1)v_j)=2\pi\left(2-m_f-2g_{\widetilde{P}}+\sum_{j=1}^{g}\operatorname{Re}\,\mathrm{ Res}_{p_j}(\nabla)\right).
\]
Consequently,
\[
\sum_{j=1}^{s}(\pi-(\rho_j+1)v_j)=2\pi\left(2-m_f-2g_{\widetilde{P}}+\sum_{j=1}^{g}\operatorname{Re}\,\mathrm{ Res}_{p_j}(\nabla)\right),
\]
as claimed.
\end{proof}

\begin{corollary}
   \label{equalityfuchsianP1}
  Let $\nabla$ be a meromorphic connection on $\mathbb{P}^1(\mathbb{C})$, with poles $\{p_0=\infty,p_1,...,p_r\}$, and set $S=\mathbb{P}^1(\mathbb{C})\setminus\{p_0,...,p_r\}\subseteq\mathbb{C}$.  Let $R_0\subset\mathbb{P}^1(\mathbb{C})$ be an $s-$sided Fuchsian geodesic polygon with vertices $q_1,...,q_s$ so that $\rho_j:=\operatorname{Res}_{q_j}\nabla>-1$. For $j=1,...,s$  let $v_j\in[0,2\pi)$ be the internal angle in $q_j$ and let $\{p_1,...,p_g\}$ be the poles of $\nabla$ contained in $R_0$.
Then
\begin{equation}\label{geodesicequalityfuch}
  \sum_{j=1}^s\left(\pi-v_j({\rho_j+1})\right)=2\pi\left(1+\sum_{i=1}^{g}\operatorname{Re}\,\mathrm{ Res}_{p_i}\nabla\right).
\end{equation}
\end{corollary}

\begin{remark}
If the boundary of $P$ in Theorem~\ref{t:attf} is regular then the formula \eqref{genTeich} reduces to \eqref{geodesicequality2}. If $\nabla$ is Fuchsian with residues in $\frac{1}{2}\mathbb{Z}$ then Corollary \ref{equalityfuchsianP1} gives Teichm{\"u}ller's lemma (see \cite[Theorem 14.1]{SK}). This is the reason why we call 
\eqref{genTeich} generalized Teichm{\"u}ller formula. 
\end{remark}

As a consequence of the generalized Teichm{\"u}ller formula we have the following corollary.

\begin{corollary}\label{ikkitaliktenglik}
  Let $\nabla$ be a meromorphic connection on a compact Riemann surface $S$, with poles $\{p_0=\infty,p_1,...,p_r\}$. Let $P\subset S$ be a simply connected part 
  of~$S$ with connected boundary given by a Fuchsian $2$-geodesic cycle with vertices $z_0$ and $z_1$. Let $v_j\in(0,2\pi)$ the internal angle at $z_j$ and assume that $\operatorname{Res}_{z_j}\nabla>-1$ for $j=0$,~$1$.
 Let $\{p_1,...,p_g\}$ be the poles of $\nabla$ contained in the interior of $P$. Then
\begin{equation}\label{2polygonequality3}
  (\operatorname{Res}_{z_0}\nabla+1)v_0+(\operatorname{Res}_{z_1}\nabla+1)v_1=-2\pi\sum_{j=1}^{g}\operatorname{Re}\,\mathrm{ Res}_{p_j}(\nabla).
\end{equation}
\end{corollary}

\begin{proof}
  It follows from \eqref{genTeich} with $s=2$ because the filling of $P$ has genus zero.
\end{proof}

\subsection{Uniqueness of geodesics}\label{uniquenessofgeodesicsss}
In this section we shall prove a result about uniqueness of geodesics in a given homotopy class.

\begin{corollary}\label{rhokattanoldandauniqueness}
  Let $\nabla$ be a Fuchsian meromorphic connection in a simply connected domain $D$ with non-negative residues. For $i=0,1,$ let $\sigma_i\colon [0,l_i]\to D$ be a smooth simple curve such that $\sigma_0(0)=z_0=\sigma_1(0)$ and $\sigma_0(l_0)=z_1=\sigma_1(l_1)$. If $\sigma_0\mid_{(0,l_0)}$ and  $\sigma_1\mid_{(0,l_1)}$ are geodesics for $\nabla$   then $\mathrm{supp}(\sigma_0)= \mathrm{supp}(\sigma_1)$. In  other words, any two points of $D$ can be joined by at most one simple geodesic arc.
\end{corollary}

\begin{proof} 
Assume, by contradiction, that $\mathrm{supp}(\sigma_0)\ne \mathrm{supp}(\sigma_1)$. Without loss of generality we can suppose that $\sigma_0$ and $\sigma_1$ do not intersect except at endpoints.
Let $\{p_1,...,p_g\}$ be the poles of $\nabla$ contained in the simply connected part $P$ bounded by $\sigma_0$ and $\sigma_1$ and let $v_j\in(0,2\pi)$ be the internal angle at $z_j$ for $j=0$, $1$. By Corollary \ref{ikkitaliktenglik} we have
\begin{equation}\label{uniquegatenglik}
  (\operatorname{Res}_{z_0}\nabla+1)v_1+(\operatorname{Res}_{z_1}\nabla+1)v_2=-2\pi\sum_{j=1}^{g}\mathrm{ Res}_{p_j}(\nabla)\;.
\end{equation}
Since $\nabla$ has only non-negative residues, this is impossible.
\end{proof}

\begin{theorem}
  Let $\nabla$ be a Fuchsian meromorphic connection with  non-negative residues on a compact Riemann surface $S$ with genus $g\ge1$. For $i=0$,~$1$, let $\sigma_i\colon[0,1]\to S$ be smooth simple curves such that $\sigma_0(0)=p_0=\sigma_1(0)$ and $\sigma_0(l_0)=p_1=\sigma_1(l_1)$. Assume $\sigma_0\mid_{(0,l_0)}$ and  $\sigma_1\mid_{(0,l_1)}$ are geodesics for $\nabla$. If $\sigma_1$ and $\sigma_2$ are homotopic then $\mathrm{supp}(\sigma_0)= \mathrm{supp}(\sigma_1)$.
\end{theorem}

\begin{proof}
Let $\pi\colon\tilde{S}\to S$ be the universal covering surface of $S$ and choose a point $z_0$ in $\pi^{-1}(p_0)$. Since the genus of $S$ is greater than $0$, we know that $\tilde{S}$ is either $\mathbb{C}$ or the unit disc. By the monodromy theorem, the lifts $\tilde{\sigma}_0$ and $\tilde{\sigma}_1$ of $\sigma_0$ and $\sigma_1$, respectively,  with initial points $z_0$ have the same terminal point $z_1$ above $p_1$. Moreover, $\tilde{\sigma}_0\mid_{(0,l_0)}$ and  $\tilde{\sigma}_1\mid_{(0,l_1)}$ are geodesics of the lift $\tilde{\nabla}$ of $\nabla$. Therefore the previous corollary implies $\mathrm{supp}(\tilde{\sigma}_0)= \mathrm{supp}(\tilde{\sigma}_1)$ and hence $\mathrm{supp}(\sigma_0)= \mathrm{supp}(\sigma_1)$.
\end{proof}

\section{Poincaré–Bendixson theorems}\label{poincare-bendexsonsection}

In this section we shall study the $\omega$-limit sets of simple geodesics of a meromorphic connection with real periods.

Let $\nabla$  be a meromorphic connection on a Riemann surface $S$ and let  $G$  be its monodromy group.  Fix  a Leray $G$-atlas $\{(U_\alpha,z_\alpha)\}$ on $S^o$ (see Definition~\ref{def:G-atlas}; in particular, the transition functions are of the form $z_\beta=a_{\alpha\beta}z_\alpha+b_{\alpha\beta}$ with $a_{\alpha\beta}\in G$ and $b_{\alpha\beta}\in\mathbb{C}$. Recall that the local representation of $\nabla$ on any chart of the atlas is identically zero. If $\sigma\colon[0,\varepsilon)\to U_\alpha$ is a geodesic for $\nabla$, then
\[
\arg dz_\alpha(\sigma'(t))=\mathrm{const.},
\]
because  
$dz_\alpha(\sigma'(t))$ is a non-zero constant. Set
\[
\mathrm{Arg}_{\sigma'}^\alpha\nabla:= \frac{dz_\alpha(\sigma'(t))}{|dz_\alpha(\sigma'(t))|}\in S^1\;.
\]
This constant depends on $(U_\alpha,z_\alpha)$ and $\sigma$.

Let now $\sigma\colon[0,\varepsilon)\to S$ be any geodesic for $\nabla$. Let  $U_\alpha$ and $U_\beta$ be charts such that  $U_\alpha\cap U_\beta \cap\mathrm{supp}(\sigma)\neq\emptyset$. By the form of the transition functions there exists $a_{\alpha\beta}\in G$ such that
\[
dz_\beta(\sigma'(t))=a_{\alpha\beta} dz_\alpha(\sigma'(t))
\]
for $\sigma(t)\in U_\alpha\cap U_\beta$. Consequently,
\begin{equation}\label{argdatenglik}
  \mathrm{Arg}_{\sigma'}^{\beta}\nabla=\frac{a_{\alpha\beta}}{|a_{\alpha\beta}|}\mathrm{Arg}_{\sigma'}^{\alpha}\nabla.
\end{equation}
Thus changing the chart amounts to multiplying $\mathrm{Arg}_{\sigma'}^{\alpha}\nabla$ by an element of the group
\[
G_1=\left\{\frac{a}{|a|}\biggm|a\in G\right\}.
\]

\begin{definition}
Set 
\[
\mathrm{Arg}_{\sigma'}\nabla:=\left[\mathrm{Arg}_{\sigma'}^{\alpha}\nabla\right]\in S^1/G_1\;;
\]
by the previous comments we see that $\mathrm{Arg}_{\sigma'}\nabla$ is well defined, though in principle it depends on the fixed Leray $\nabla$-atlas.

Given $[\theta]\in S^1/G_1$, if $\mathrm{Arg}_{\sigma'}\nabla=[\theta]$ we shall say that $[\theta]$ is the \emph{direction} of a geodesic $\sigma\colon[0,\varepsilon)\to S$, or that $\sigma$ has  $\theta$\emph{-direction}.
\end{definition}

\begin{remark}
The direction of a given geodesic depends on the chosen atlas. On the other hand, from the definition it follows that for two geodesics having the same direction is a notion independent of the specific Leray $G$-atlas chosen.
\end{remark}

\subsection{Ring domains}\label{ringdomains}
Let us first introduce the notion of ring domain.

\begin{definition}
Let $\nabla$ be a meromorphic connection on a Riemann surface $S$. Set $S^o=S\setminus\Sigma$, where $\Sigma$ is the set of poles of $\nabla$. A connected open set $R\subset S^o$ with $\pi_1(R)=\mathbb{Z}$ is said to be a \emph{ring domain} if there are simple periodic  geodesics $\sigma_a:\mathbb{R}\to S^o$  for any $a\in(0,1)$ such that
    $\sigma_{a_1}$ and $\sigma_{a_2}$ have no common points for $a_1\ne a_2$, their support $\mathrm{supp}(\sigma_a)$ depends continuous on $a$ in the Hausdorff topology, and
\[
R=\bigcup_{a\in(0,1)}\mathrm{supp}(\sigma_a).
\]
We say $\sigma_a$ is a \emph{leaf} of $R$. We denote the foliation given by the $\sigma_a$ by $\mathcal{R}$.

Assume now that $\nabla$ has real periods and let $g$ be a flat metric on $S^o$ adapted to $\nabla$. The $g-$\emph{width} of  the ring domain $R$ is
\[
\sup\{\mathrm{dist}_g^R\bigl(\mathrm{supp}(\sigma_{a_1}),\mathrm{supp}(\sigma_{a_2})\bigr)\bigm| a_1,a_2\in(0,1)\},
\]
where $\mathrm{dist}_g^R$ is the distance induced by $g$ on $R$. 
\end{definition}

\begin{example}
Let $\nabla$ be the meromorphic connection on the unit disc $\mathbb{D}=\{z\in\mathbb{C}\mid |z|<1\}$ with the local representation
$\eta=-\frac{1}{z}dz$.  Then for any $0\le r_1<r_2\le 1$ the domain
\[
R=\{z\in\mathbb{D}\mid r_1<|z|<r_2\}
\] 
is a ring domain for $\nabla$. Indeed, it is not difficult to check that, for $r\in(0,1)$, the curve $\sigma_r\colon\mathbb{R}\to\mathbb{D}$ given by
\[
\sigma_r(t)=re^{it},
\]
is a periodic geodesic for $\nabla$. The singular flat metric $g$ adapted to $\nabla$ is given by 
\[
g^{\frac12}=\frac{|dz|}{|z|}.
\]
Then the $g-$width of $R$ is equal to $\log r_2-\log r_1$. In particular, $R$ has infinite $g$-width when $r_1=0$.

Assume now that $r_1>0$ and set $U:=R\setminus\mathbb{R}^+$. Then the local isometry $J(z)=\log z$ is well defined on~$U$. Moreover, $g^{\frac{1}{2}}=|dJ|$ on $U$. It is not difficult to see that
\[
J(U)=\{w\in \mathbb{C}\mid \log r_1<\operatorname{Re}w<\log r_2,0<\operatorname{Im}w<2\pi\}
\]
 is a rectangle. Moreover, the $J$-image of any geodesic of $\nabla$ in $U$ is Euclidean segment, because $J$ is a local isometry. In particular, we have $J(\sigma_r)=\{\operatorname{Re}w=\log r, 0<\operatorname{Im}w<2\pi\}$.
\end{example}

Now we study the image of a ring domain under local isometries of a meromorphic connection $\nabla$. 

\begin{lemma}\label{euclideanrectangle}
  Let $\nabla$ be a meromorphic connection on a Riemann surface $S$. Assume  $\nabla$ has real periods. Set $S^o:=S\setminus\Sigma$, where $\Sigma$ is the set of poles of $\nabla$.  Let $g$ be a flat metric adapted to $\nabla$ on $S^o$. Let $R$ be a ring domain with $g$-width equal to $r<\infty$. Assume $\partial R=\gamma_1\cup\gamma_2$ with $\gamma_1$ and $\gamma_2$  disjoint connected sets. Assume there exist $p_1\in\gamma_1$ and $p_2\in\gamma_2$ and a simple geodesic $\beta:[0,1]\to \overline{R}$ connecting $p_1$ and $p_2$ with $g-$length equal to $r$. Let $J$ be a local isometry of $g$ on $U:=R\setminus\mathrm{supp}(\beta)$. Then $J(U)$ is a rectangle. In particular all leaves of $R$ have the same $g$-length.
\end{lemma}

\begin{proof}
Let $\{U_j\}_{j=1}^3$  be an open cover for $R$ such that:
\begin{itemize}
\item both $U_j$ and $R\setminus U_j$ are simply connected for $j=1$, $2$, $3$;
\item $U_i\cup U_j$ and $U_i\cap U_j$ are simply connected for $i$,  $j=1$, $2$, $3$;
\item $U_1\cap U_2\cap U_3=\emptyset$;
\item $\operatorname{supp}(\beta)\subset U_1\cap U_3$.
\end{itemize}
Let $V_0=(U_1\cup U_2)\setminus\overline{U_3}\subset U$ and $J_0=J|_{V_0}$. Let $J_1$ be the analytic continuation of $J_0$ on $V_1=U_1\cup U_2$ and $J_2$ the analytic continuation of $J_0$ on $V_2=U_2\cup U_3$. Note that for every leaf $\sigma$ of~$R$ we have that $J_j(\sigma\cap V_j)$ is an Euclidean segment in $J_j(V_j)$; moreover, the $J_j$-images of different leaves are disjoint and their union is the simply connected domain $J_j(V_j)$. Moreover, $\Gamma_\sigma:=J_1(\sigma\cap V_1)\cup J_2(\sigma\cap V_2)$ is an Euclidean segment in $D:=J_1(V_1)\cup J_2(V_2)$ and $D=\cup_{\sigma\in\mathcal{R}}\Gamma_\sigma$; in other words, $D$ is foliated by the Euclidean segments $\Gamma_\sigma$ with $\sigma\in\mathcal{R}$. 
Put $\beta_j=J_j(\beta)$ for $j=1$,~$2$. Then $D\setminus(\beta_1\cup\beta_2)$ has three connected components; one is $J(U)$. 

By construction, $\overline{J(U)}=J(U)\cup\beta_1\cup \beta_2$ can be written as union of non-intersecting Euclidean segments starting in~$\beta_1$ and ending in $\beta_2$. Notice that $\beta_1$ and $\beta_2$, being images of a geodesic of length~$r$ via a local isometry, are Euclidean segments of length~$r$. Since $J_1|_{U_2}-J_2|_{U_2}$ is a constant, $\beta_1$ and $\beta_2$ are parallel; hence $\overline{J(U)}$ is a parallelogram. But, by definition of $g$-width of a ring domain, the Euclidean height of $\overline{J(U)}$ is equal to~$r$; since the Euclidean length of $\beta_1$ and $\beta_2$ is also equal to~$r$, it follows that $\overline{J(U)}$ must be a rectangle. Finally, since the $J$-images of the leaves of $R$ are parallel Euclidean segments joining $\beta_1$ and $\beta_2$ foliating~$J(U)$, it follows that all leaves of $R$ have the same $g$-length.
\end{proof}

By definition  a ring domain is a union of simple periodic geodesics. In the next lemma we describe the possible boundaries of ring domains.

\begin{lemma}\label{boundaryring}
Let $\nabla$ be a Fuchsian meromorphic connection on a Riemann surface $S$ with real periods. Set $S^o=S\setminus\Sigma$, where $\Sigma$ is the set of poles of $\nabla$.  Let $R$ be a ring domain relatively compact in $S$. Assume $\partial R$ is not empty. Let $\gamma_1$ be a maximal connected component of $\partial R$. Then $\gamma_1$  is either
\begin{enumerate}
     \item a pole of $\nabla$ with residue $-1$;  or
     \item the support of a periodic 
     geodesic; or
     \item a graph of  saddle connections.
 \end{enumerate}
\end{lemma}

\begin{proof}
Assume $\gamma_1=\{p\}$ is a single point. Then every neighbourhood of $p$  contains a periodic geodesic. In particular, there exists a simply connected part containing $p$ and no poles other than $p$ and whose boundary is a periodic geodesic. Recalling that, by assumption, $\nabla$ has real periods (and hence real residues), we can then apply \eqref{genTeich} with $s=0$, $m_f=1$, $g_{\widetilde{P}}=0$ and $g=1$ to deduce that 
$p$ must be a pole with residue $-1$.

Assume now that $\gamma_1$ is not a single point; then it must contain a regular point $p\in\gamma_1\cap S^0$. Let $(U,z)$ be a $\nabla-$chart centered at $p$. Then the leaves of  $R$ in $z(U\cap R)$ are parallel Euclidean segments. Hence $z(U\cap \gamma_1)$ is a Euclidean segment. Consequently, $\gamma_1$ is locally the support of a geodesic. Since $\gamma_1$ is compact, each component of  $\gamma_1\cap S^o$ is a geodesic.
If $\gamma_1$ contains no poles then it is the support of a closed or periodic geodesic; 
since $\nabla$ has real periods, Lemma \ref{l:closedperiodic} implies that $\gamma_1$ is the support of a periodic geodesic. 
Finally,
if $\gamma_1$ contains poles then it is a graph of saddle connections.
\end{proof}

 In the next lemma  we study the possible boundary of \textit{maximal} ring domains, i.e., ring domains $R$ which are not a proper subset of another ring domain.
 
\begin{lemma}\label{partialring}
 Let $\nabla$ be a meromorphic connection on a Riemann surface $S$ with real periods; assume that the residues are all greater than $-1$. Set $S^o=S\setminus\Sigma$, where $\Sigma$ is the set of poles of $\nabla$. Assume there exists a simple periodic geodesic $\sigma\colon [0,\varepsilon)\to S^o$. Then there exists a ring domain $\tilde{R}$  containing $\sigma$ as a leaf.  
  
Furthermore, if the maximal ring domain $R$ containing $\sigma$ as a leaf has nonempty boundary then its boundary is a boundary graph of saddle connections.
Moreover, any saddle connection in $\partial R$ has the same direction as $\sigma$.
\end{lemma}

\begin{proof}
Let $g$  be the singular flat metric adapted to $\nabla$.  Choose $r_0$ small enough such that for all $0<r\le r_0$ we have:
\begin{itemize}
\item for any  $q\in \mathrm{supp}(\sigma)$ the $g$-ball $B_{g}(q,3r):=\{p\in S\mid\mathrm{dist}_g(q,p)<3r\}$ is simply connected; and 
\item the boundary of the tubular neighbourhood 
\[
U_{r}(\sigma):=\{p\in S\mid \mathrm{dist}_g(p,\mathrm{supp}(\sigma) )<r\}
\]
of radius $r$ of the curve $\sigma$ is composed by two different smooth curves $\sigma_i\colon [0,1]\to S$ for $i=1$,~$2$, so that
$\mathrm{dist}_g(\mathrm{supp}(\sigma_1),\mathrm{supp}(\sigma_2))=2r$.
\end{itemize}

Pick a point $p_1\in\mathrm{supp}(\sigma)$. Take a chart $(B_g(p_1,2r),z)$ with $r\le r_0$.
Let $J$ be a local isometry of  $\nabla$ on $B_g(p_1,2r)$.  Then $J(\mathrm{supp}(\sigma)\cap B_g(p_1,2r))$ is a Euclidean segment. On the other hand, the Euclidean distance from any point of $J(\mathrm{supp}(\sigma_1)\cap B_g(p_1,2r))$ to $J(\mathrm{supp}(\sigma)\cap B_g(p_1,2r))$ is equal to $r$. Hence  $J(\mathrm{supp}(\sigma_1)\cap B_g(p_1,2r))$ is a Euclidean segment. Consequently, $\mathrm{supp}(\sigma_1)\cap B_g(p_1,2r)$ is the support of a geodesic. Hence $\mathrm{supp}(\sigma_1)$ is the support of a closed or periodic geodesic. Since $\nabla$ has real periods and real residues, Lemma \ref{l:closedperiodic} implies that $\mathrm{supp}(\sigma_1)$ is the support of a periodic geodesic. 
Hence for any $0<r\le r_0$ we have periodic geodesics $\sigma_{r}^1$ and $\sigma_{r}^2$.  Then
\[
R=\sigma\cup \bigcup_{j=1}^2\bigcup_{0<r<r_0} \sigma_{r}^j
\]
is a ring domain containing $\sigma$ as a leaf.

Let  now $\tilde R$ be  the maximal ring domain containing $\sigma$  as a leaf. Let $\gamma_1$ be a connected component of $\partial R$. Assume $\gamma_1$ is a single point. Then, by Lemma \ref{boundaryring}, it must be a pole with residue $-1$, impossible because we assumed that all residues should be larger than $-1$.

 Assume now that $\gamma_1$ contains no poles. Then, by the previous Lemma, $\gamma_1$ is the support of a periodic 
 geodesic; but in this case, using the previous construction, we might enlarge $\tilde R$, again the maximality of $\tilde R$.  Consequently, $\partial\tilde R$ is a graph of saddle connections.

Finally, assume $\sigma$ has direction $[\theta]$ with respect to a $\nabla$-atlas. Since the image of any leaf $\sigma_a$ of $\tilde R$ under a local isometry is locally parallel to the image of $\sigma$, it follows that $\sigma_a$ has the same direction as $\sigma$. Let $p\in \partial R$ be a regular point for $\nabla$. Let $(U,z)$ be a $\nabla-$chart centered at $p$. Again, since the leaves of  $R$ in $z(U\cap R)$ are parallel Euclidean segments, it follows that $z(U\cap \partial R)$ is parallel to these Euclidean segments. Hence if $\gamma$ is a saddle connection with $\mathrm{supp}(\gamma)\subset\partial R$ then $\gamma$ must have the same direction as $\sigma$.
\end{proof}

\begin{corollary}\label{horizontalsaddlecon}
Let $\nabla$ be a meromorphic connection on a compact Riemann surface $S$ with real periods; assume that all the residues are greater than $-1$. Set $S^o=S\setminus\Sigma$, where $\Sigma$ is the set of poles of $\nabla$. Assume there exists a simple periodic geodesic $\sigma\colon [0,\varepsilon)\to S^o$. If $\nabla$ has at least one pole then there exists a saddle connection having the same direction as $\sigma$.
\end{corollary}

\begin{proof}
  Let $R$ be the maximal ring domain containing $\sigma$ as a leaf. Since $\Sigma$ is not empty, $\partial R$ cannot be empty. Then, by previous lemma, $\partial R$ contains a  saddle connection with the same direction as $\sigma$.
\end{proof}

Now we study maximal ring domains with empty boundary.

\begin{lemma}
 Let $\nabla$ be a meromorphic connection on a compact Riemann surface $S$ with real periods; assume that all residues are greater than $-1$. Set $S^o=S\setminus\Sigma$, where $\Sigma$ is the set of poles of $\nabla$. Assume there exists a simple periodic geodesic $\sigma\colon [0,\varepsilon)\to S^o$. Let $R$ be the maximal ring domain containing $\sigma$ as a leaf. Then $\partial R$ is empty if and only if $S$ is a torus and $\nabla$ is a holomorphic connection with real periods.
\end{lemma}

\begin{proof}
 If $\partial R$ is empty then $R=S$. Hence $\nabla$ is a holomorphic connection and $S$ is torus.
 
On the other hand assume $S$ is torus and $\nabla$ is holomorphic. Assume, by contradiction, that $\partial R$ is not empty. Then, by Lemma~\ref{partialring}, $\partial R$ would contain at least one pole. But this is impossible, because $\nabla$ is a holomorphic connection.
\end{proof}

\subsection{Periodic geodesics}
In this subsection we show that the $\omega$-limit set of a simple non-periodic geodesic of a meromorphic connection with real periods cannot be a periodic geodesic. 

\begin{lemma}\label{finitinters}
    Let $\nabla$ be a Fuchsian meromorphic connection on a Riemann surface $S$ with real periods. Set $S^o:=S\setminus\Sigma$, where $\Sigma$ is the set of poles of $\nabla$. Let $\sigma\colon [0,\varepsilon)\to S^o$ be a simple not closed geodesic for~$\nabla$ such that the $\omega$-limit set $W$ of $\sigma$ is a closed (or periodic) geodesic or a boundary graph of saddle connections. Then $\sigma$ does not intersect $W$.
\end{lemma}

\begin{proof}
Our goal is to show that $W\cap \mathrm{supp}(\sigma)=\emptyset$. Assume, by contradiction, that $W\cap \mathrm{supp}(\sigma)$ is not empty and take $p_0\in W\cap \mathrm{supp}(\sigma)$. Notice that $p_0$ is not a pole. Furthermore, $\sigma$ and $W$ are transversal at~$p_0$, because otherwise we would have $\operatorname{supp}(\sigma)\subseteq W$ and this is impossible, because $\sigma$ is not closed nor is a saddle connection. 

Let $(U,z)$ be a $\nabla$-chart at $p_0$ such that $\gamma:=W\cap U$ is connected.  Notice that $z(\gamma)$ is a Euclidean segment. For $k\ge 1$, let $\sigma_k\colon (\varepsilon_{k-1},\varepsilon_{k})\to S^o$ be the geodesics of $\nabla$, with $0\le\varepsilon_0<\varepsilon_k<\varepsilon_{k+1}<\varepsilon$,  such that $\mathrm{supp}(\sigma_k)$ is a connected component of $ \mathrm{supp}(\sigma)\cap U$. Since $\gamma$ is in the $\omega$-limit set of $\sigma$,  the supports of $\sigma_k$ must accumulate $\gamma$. Let $\sigma_{k_0}$ be the segment of $\sigma$ intersecting transversally $\gamma$ in $p_0$. Since $z(\sigma_k)$ accumulates to $z(\gamma)$, there exists $k_1>k_0$ such that $z(\mathrm{supp}(\sigma_{k_1}))\cap z(\mathrm{supp}(\sigma_{k_0}))\ne \emptyset$. Since $\sigma$ is simple, we have a contradiction.
\end{proof}

\begin{proposition}\label{periodicyoqekan}
Let $\nabla$ be a meromorphic connection on a Riemann surface $S$ with real periods.  Set $S^o:=S\setminus\Sigma$, where $\Sigma$ is the set of poles of $\nabla$.  Let $\sigma\colon [0,\varepsilon)\to S^o$ be a simple non-periodic geodesic for $\nabla$. Then the  $\omega$-limit set of $\sigma$ cannot be the support of a periodic geodesic.
\end{proposition}

\begin{proof}
  Assume the $\omega$-limit set of  $\sigma$ is the support of a periodic geodesic $\gamma_1\colon \mathbb{R}\to S^o$.   Thanks to Lemma \ref{finitinters}, $\sigma$ does not intersect $\gamma_1$. Let $g$ be a flat metric adapted to $\nabla$ on $S^o$. Let $R\subset S$ be a ring domain with finite $g$-width containing $\gamma_1$ as leaf, given by Lemma~\ref{partialring}. Since the $\omega$-limit set of  $\sigma$ is $\mathrm{supp}(\gamma_1)$ there exists $\varepsilon_0\in[0,\varepsilon)$ such that $\sigma(t)\in R$ for all $t\in[\varepsilon_0,\varepsilon)$. Let $R_1\subset R$ be a ring domain such that $\sigma(t)\in R_1$ for all $t\in[\varepsilon_0,\varepsilon)$ and $\partial R_1\supset\mathrm{supp}(\gamma_1).$ More precisely, we can choose as $R_1$ the connected component of $R\setminus\gamma_1$ eventually containing $\sigma$.
  Let $r$ be the $g$-width of $R_1$. Let
\[
\partial R_1=\mathrm{supp}(\gamma_1)\cup\mathrm{supp}(\gamma_2(r))
\]
for some periodic geodesic $\gamma_2(r)$. Fix $p_1\in \gamma_1$. Choose $p_2:=p_2(r)\in \gamma_2(r)$  such that $\text{dist}_g(p_1,p_2)=r$. We choose $r$ small enough so that there exists a geodesic $\beta:[0,1]\to \overline{R}_1$ so that $\beta(0)=p_1$ and $\beta(1)=p_2$  and $\mathrm{length}_g(\beta)=r$. Let $(U,z)$ be a chart with $U:=R_1\setminus \mathrm{supp}(\beta)$. Let $J$ be a local isometry of $\nabla$ on $U$.  By Lemma \ref{euclideanrectangle}, $J(U)$ is a Euclidean rectangle. We shall call ``vertical" the two sides transversal to the image via~$J$ of $\gamma_1$.

Since $\sigma$ accumulates $\gamma_1$, it must intersect $\beta$ infinitely many times. Moreover, the intersections are transversal, because $\beta$ is a geodesic different from~$\sigma$, and they accumulate only at~$p_1$; in particular, they are countable. Therefore we can find a strictly increasing sequence $\{\varepsilon_j\}\subset[\varepsilon_0,\varepsilon)$
such that 
\[
\mathrm{supp}(\sigma)\cap\mathrm{supp}(\beta)=\{\sigma(\varepsilon_j)\mid j\in\mathbb{N}\}.
\] 
Put $\sigma_j:=\sigma\mid_{(\varepsilon_j,\varepsilon_{j+1})}$. Then $\sigma_j$ is a maximal geodesic arc in $U$ tending to $\mathrm{supp}(\beta)$ both in backward and in forward time. Since $\sigma_j$ is a geodesic segment, $J(\mathrm{supp}(\sigma_j))$ is an Euclidean segment. Since $J(U)$ is a rectangle, the Euclidean segment $J(\mathrm{supp}(\sigma_j))$ intersects with the same angle the two vertical sides of $J(U)$. Since $\sigma_{j+1}$ is the continuation of $\sigma_j$ we see that $J(\mathrm{supp}(\sigma_j))$ and $J(\mathrm{supp}(\sigma_{j+1}))$ must intersect the vertical sides of $J(U)$ at the same angle. Hence  $J(\mathrm{supp}(\sigma_j))$ and $J(\mathrm{supp}(\sigma_{j+1}))$ are parallel. 

By assumption, we know that $\{J(\mathrm{supp}(\sigma_j) )\}_{j=1}^{\infty}$ accumulates to $l_{\gamma_1}\subset \partial J(U)$, a side of the rectangle $J(U)$. Hence $l_{\gamma_1}$ is parallel to  $J(\mathrm{supp}(\sigma_j))$, because otherwise there would exist $j_0$ such that $J(\mathrm{supp}(\sigma_{j_0}))$ intersect $l_{\gamma_1}$. Since $U\subset R_1$, any maximal Euclidean segment $l_1\subset J(U)$ which is parallel to $l_{\gamma_1}$ is the image of the support of a periodic geodesic; hence  $\overline{\mathrm{supp}(\sigma_j)}$ is the support of a periodic geodesic. Since $\sigma$ is not a periodic geodesic and $\sigma_j$ is a part of it we have a contradiction. Hence, the $\omega$-limit set of $\sigma$ cannot be a periodic geodesic.
\end{proof}

We explicitly remark that this proposition holds for any meromorphic connection $\nabla$ with real periods, without any limitation of the value of the residues.

\begin{corollary}
  Let $\nabla$ be a meromorphic connection on $\mathbb{P}^1(\mathbb{C})$. Set $S^o:=\mathbb{P}^1(\mathbb{C})\setminus\Sigma$, where $\Sigma$ 
  is the set of poles of $\nabla$. Let $\sigma\colon[0,\varepsilon)\to S^o$ be a simple non-periodic geodesic for $\nabla$. Then the  $\omega$-limit set of $\sigma$ cannot be the support of a periodic geodesic.
\end{corollary}

\begin{proof}
  Assume the $\omega$-limit set of  $\sigma$ is the support of a periodic geodesic $\gamma_1:\mathbb{R}\to S^o$. Note that since $\sigma$ is simple $\gamma_1$ cannot intersect itself.   Hence $\gamma_1$ is a simple periodic geodesic
  and, by \cite[Corollary 4.5]{AT1}, it surrounds poles $p_1$,~$p_2,\ldots, p_g$ with
  \begin{equation}\label{yig`indi-1}
    \sum_{j=1}^{g}\operatorname{Res}_{p_j}\nabla=-1.
  \end{equation}
 Without loss of generality we can assume that $\gamma_1$ and its interior are contained in~$\mathbb{C}$.
  
Let $U:=U_r(\gamma_1)\subset\mathbb{C}$ be a tubular $r$-neighbourhood (with respect to the Euclidean metric in $\mathbb{C}$) of $\gamma_1$ such that $U\cap\Sigma=\emptyset$. Then $U$ is a doubly connected domain eventually containing~$\sigma$. By Proposition~\ref{periodicyoqekan}, it is enough to show that $\nabla$ has real periods in $U$. 

The fundamental group of $U$ is generated by~$\gamma_1$; moreover, \cite[Proposition 3.6]{AT1} implies that the monodromy representation is given by
\[
\rho(\gamma_1)=\exp\left(\int_{\gamma_1}\eta\right)\;,
\]
where $\eta$ is the local representation of~$\nabla$ on~$U$. Now, let $D$ be the union of~$U$ with the interior of~$\gamma_1$, so that $D$ is a simply connected domain in~$\mathbb{C}$ containing $p_1,\ldots, p_g$. The form $\eta$ then extends to a meromorphic form on~$D$, still denoted by~$\eta$, with poles $p_1,\ldots,p_g$. Since $\gamma_1$ is a simple closed curve surrounding $p_1,\ldots,p_g$ we have
\[
\frac{1}{2\pi i}\int_{\gamma_1}\eta=\sum_{j=1}^g \operatorname{Res}_{p_j}\nabla\;;
\]
therefore recalling \eqref{yig`indi-1} we get $\rho(\gamma_1)=\exp(-2\pi i)=1$, that is, the monodromy group of~$\nabla$ on~$U$ is trivial and hence $\nabla$ has real periods on~$U$, as required.
\end{proof}

\begin{corollary}
   Let $\nabla$ be a meromorphic connection on $\mathbb{P}^1(\mathbb{C})$. Set $S^o:=\mathbb{P}^1(\mathbb{C})\setminus\Sigma$, where $\Sigma$ 
   is the set of poles of $\nabla$. Let $\sigma\colon[0,\varepsilon)\to S^o$ be a simple non-periodic geodesic for $\nabla$. If the  $\omega$-limit set of $\sigma$ is the support of a closed geodesic $\gamma$ then $\gamma$ surrounds poles $p_1,p_2,...,p_g$ with
  $$
    \sum_{j=1}^{g}\mathrm{Re\,Res}_{p_j}\nabla=-1
  $$
  and
  $$
    \sum_{j=1}^{g}\mathrm{Im\,Res}_{p_j}\nabla\ne0.
  $$
\end{corollary}
\begin{proof}
  It follows by the previous corollary and \cite[Corollary 4.5]{AT1}.
\end{proof}

There are examples of simple non-periodic geodesics in $\mathbb{P}^1(\mathbb{C})$ whose $\omega$-limit set is a closed non-periodic geodesic; see \cite{AT1}.

\subsection{Boundary graph of saddle connections}

By Theorem \ref{t3}, another possible $\omega$-limit set of a simple maximal geodesic of a meromorphic connection $\nabla$ on a compact Riemann surface $S$ is a boundary graph of saddle connections. In this section we shall prove that this cannot happen for a Fuchsian meromorphic connection with real periods.

As we have seen in Lemma \ref{partialring}, any simple periodic geodesic is a leaf of a ring domain. Now we prove a similar property for a boundary graph of saddle connections.

\begin{lemma}\label{pbl2} 
Let $\nabla$ be a Fuchsian meromorphic connection on a Riemann surface $S$ with real periods. Set $S^o=S\setminus\Sigma$, where  $\Sigma$ is the set of poles of $\nabla$. Let $\Gamma\subset S$ be a boundary graph of saddle connections, and choose a connected open set $V$ such that $\Gamma$ is a connected component of~$\partial V$. Assume $\Gamma$ is the $\omega$-limit set of a simple geodesic $\sigma\colon [0,\varepsilon)\to S^o$; then there exists a ring domain $R\subset V$ such that $ \partial R=\Gamma\cup \gamma$, where $\gamma$ is the support of a periodic geodesic.
\end{lemma}

\begin{proof}  Since $\nabla$ has real periods there  exists a singular flat metric $g$ on $S$ adapted to $\nabla$. Let $W\subset S$ be a connected open set containing $\Gamma$ such that 
\[
W\cap \Sigma=\Gamma\cap\Sigma=:\Sigma_1.
\]
By Corollary \ref{boundaryburchak}, for any $p\in \Sigma_1$ we have $\operatorname{Res}_{p}\nabla\ge-\frac12$. Hence, by Lemma \ref{2rrho+1}, $\Gamma$ has finite $g$-length.

For $p_j\in\Sigma_1$, let $(U_j,z_j)$ be a chart adapted to $(\nabla,p_j)$ such that $U_j\subset W$. Let $r_j$ be the $g$-radius of~$U_j$, i.e., the $g$-length of a critical geodesic in $U_j$ (by Lemma \ref{lengthcriticalg} all critical geodesics in $U_j$ have the same $g-$length). Choose $0<r_0<\frac16 \min r_j$ such that for any point $q\in\Gamma$ the set
\[
B_{g}(q,3r_0):=\{p\in S\mid\mathrm{dist}_g(p,q)<3r_0\}
\]
is a simply connected subset of $W$ such that $\Gamma\cap B_{g}(q,3r_0)$ is connected; such an $r_0$ exists because  $\Gamma$ is compact and locally connected. For $0<r\le r_0$ set
\[
U_{r}(\Gamma):=\{p\in V\cap W\mid \mathrm{dist}_g(p,\Gamma)<r\}.
\]
Then $\partial U_{r}(\Gamma)$  has two connected components. One is  $\Gamma$; let $\gamma(r)$ be the other one. To prove the statement it is enough to prove that $U_r(\Gamma)$ is a ring domain and that $\gamma(r)$ is the support of a periodic geodesic. 

By construction, for any $z_r\in\gamma(r)$ we have
\[
\mathrm{dist}_g(\Gamma,\gamma(r))=\mathrm{dist}_g(\Gamma,z_r)=r.
\]
Note that, by Proposition \ref{orasi2pitaqsim}, $z(U_j\cap V)=\cup _{k=1}^m C_j^k$, where $C_j^k$ is a sector with interior angle $\frac{\pi}{\rho_j+1}\le2\pi$. 
  We claim that $\gamma_j^k:=\gamma_j^k(r):=\gamma(r)\cap z^{-1}(C_j^k)$ is the support of a geodesic.  Since $(U_j,z_j)$ is a chart adapted to~$(\nabla,p_j)$ the local representation of $\nabla$ on $U_j$ is $\eta_j=\frac{\rho_j}{z_j}dz_j$. Then the map $J(z_j)=z_j^{\rho_j+1}$ is a local isometry of~$\nabla$ on $C_j^k$. Moreover,
\[
J(C_j^k)=\{w\in\mathbb{C}\mid\alpha_j^k<\arg w<\pi+\alpha_j^k,|w|<\tilde{r}_j\}
\] 
is a half disc for some $\alpha_j^k\in[0,\pi]$ and $\tilde{r}_j>0$. Put $\Gamma_j^k=\partial C_j^k\cap z(\Gamma\cap U_j)$. By construction, $J$ continuously extends to $\overline{C}_j^k$ and   $J(\Gamma_j^k)$ is the diameter of the half disc.   Since $J$ is a local isometry of $\nabla$ on $C_j^k$ and  any point of $\gamma_j^k$ has the same $g-$distance from $\Gamma_j^k$, we can see that $J(z_j(\gamma_j^k))$ is a Euclidean segment parallel to $J(z_j(\Gamma_j^k))$. Hence $\gamma_j^k$ is the support of a geodesic.

Since $r<\frac16 \min r_j$, we can find $q_k\in \Gamma$ for $k=1,\ldots,n$ such that
\begin{enumerate}
  \item there are no poles of $\nabla$ in $B_{g}(q_k,2r)$, i.e., $B_{g}(q_k,2r)\cap\Sigma_1=\emptyset$;
  \item $\gamma$ is contained in
\[
\bigcup_{k=1}^n B_{g}(q_k,2r)\cup\bigcup_j U_j.
\]
\end{enumerate}
As in the proof of Lemma \ref{partialring} we can show that $\gamma(r)\cap B_{g}(q_k,2r)$ is the support of a geodesic. Consequently, $\gamma(r)$ is the support of a closed geodesic. Since $\nabla$ has real periods $\gamma$ is the support of a periodic geodesic. 
Since this holds for all $0<r<_0$ and $U_{r_0}(\Gamma)$ is foliated by $\gamma(r)$ we have found a ring domain $R$ with $g-$width equal to $r$ such that $ \partial R=\Gamma\cup \gamma$ where $\gamma$ is the support of a periodic geodesic.
\end{proof}
 \begin{remark}\label{R:sigmainR}
   In Lemma \ref{pbl2}, we can also assume that the support of $\sigma$ is eventually contained in~$R\subseteq V$. Indeed, $\sigma$ cannot intersect $\Gamma$ by Lemma \ref{finitinters} (see also \cite[Proposition 4.1]{AB}). Therefore, $\operatorname{supp}(\sigma)$ is a subset of a connected component $C$ of $S\setminus \Gamma$ and then it suffices to take $V$ contained in $C$.
 \end{remark}

By using Lemma \ref{pbl2}, and the same technique used in the proof of Proposition \ref{periodicyoqekan} we get the following result.

\begin{proposition}\label{boundarygraphyoqekan}
Let $\nabla$ be a Fuchsian meromorphic connection on a Riemann surface $S$ with real periods. Set $S^o:=S\setminus\Sigma$, where $\Sigma$ is the set of poles of $\nabla$. Let $\sigma\colon[0,\varepsilon)\to S^o$ be a simple non-periodic geodesic for $\nabla$. Then the  $\omega$-limit set of $\sigma$ cannot be a boundary graph of saddle connections.
\end{proposition}

\begin{proof}
Assume, by contradiction, that  the  $\omega$-limit set of $\sigma$  is  a boundary graph of saddle connections~$\Gamma$. Let $V$ be  a connected open subset of~$S$ eventually containing $\sigma$ and such that $\Gamma$ is a connected component of $\partial V$. 

Thanks to Lemma \ref{finitinters}, $\sigma$ does not intersect $\Gamma$. Let $g$ be a singular flat metric adapted to $\nabla$. By Lemma \ref{pbl2}, there exists a ring domain $R\subset V$ with $g-$width equal to $r>0$ such that 
\[
\partial R=\Gamma\cup\gamma(r),
\]
where $\gamma(r)$ is a periodic geodesic for $\nabla$. Moreover, by Remark~\ref{R:sigmainR},  we can also assume that the support of~$\sigma$ is eventually contained in~$R$, i.e., there exists $\varepsilon_0\in[0,\varepsilon)$ such that $\sigma(t)\in R$ for all $t\in[\varepsilon_0,\varepsilon)$.  Notice that $R$ does not contain any pole of $\nabla$.  

 Fix $p_1\in \Gamma$. Choose $p_2:=p_2(r)\in \gamma(r)$ such that $\text{dist}_g(p_1,p_2)=r$. We can choose $r$ small enough so that there exists a geodesic segment $\beta\colon [0,1]\to \overline{R}$ such that $\beta(0)=p_1$ and $\beta(1)=p_2$  and $\mathrm{length}_g(\beta)=r$. Let $(U,z)$ be a chart with $U:=R\setminus \mathrm{supp}(\beta)$. Let $\eta$ be the representation of $\nabla$ on $U$. Let $J$ be a local isometry of $\nabla$ on $U$.  By Lemma \ref{euclideanrectangle}, $J(U)$ is a  rectangle (even if the boundary of $U$ contains poles). The rest of the proof is the same as the proof of Proposition \ref{periodicyoqekan}.
\end{proof}

\begin{proof}[{Proof of Theorem \ref{pbt1}}]
  It follows from Theorem \ref{t3} together with Propositions \ref{periodicyoqekan} and~\ref{boundarygraphyoqekan}.
\end{proof}

\appendix

\section{Transversally Cantor-like geodesic sets}\label{s:appA}
\begin{definition}
    Let $\nabla$ be a meromorphic connection on a compact Riemann surface $S$. A set $W\subset S$ with $\mathring W=\emptyset$ is said to be a \textit{transversally Cantor-like geodesic set} if the following  conditions hold:
  \begin{enumerate}
        \item there exists a maximal non self-intersecting geodesic $\sigma\colon(\varepsilon_-,\varepsilon_+)\to S^o$ such that $W$ is the closure of the support of~$\sigma$;
         \item for any non self-intersecting geodesic $\gamma\colon(-\delta,\delta)\to S^0$ transverse to $\sigma$ the intersection $\gamma|_{[-\delta/2,\delta/2]}\cap W$ is a perfect totally disconnected set (a Cantor set).
  \end{enumerate}
\end{definition}

\begin{lemma}\label{l:Acl}
     Let $\nabla$ be a meromorphic connection on a compact Riemann surface $S$. Let $\sigma\colon (\varepsilon_-,\varepsilon_+)\to S^o$ be a maximal simple geodesic and denote by  $W$ its $\omega$-limit set. Assume $\operatorname{supp}(\sigma)\subseteq W$. Then either
     \begin{enumerate}
         \item $W$ has non-empty interior; or
         \item  $\sigma$ is a closed geodesic and $W=\operatorname{supp}(\sigma)$; or
         \item $W$ is a transversally Cantor-like geodesic set.
           \end{enumerate}
\end{lemma}

\begin{proof}
Assume $W$ has empty interior. Pick a point $z_0\in \sigma$. For $\delta$ small enough let $\gamma\colon (-\delta,\delta)\to S^o$ be a geodesic transversal to $\sigma$ at $z_0$ such that $\gamma(0)=z_0$. Since $\operatorname{supp}(\sigma)\subseteq W$, then $\sigma$ must intersect $\gamma$ infinitely many times (not necessarily in distinct points). Assume that $z$ is isolated in $\operatorname{supp}(\gamma)\cap W$, i.e., there exists a neighbourhood $U$ of $z$ such that  $U\cap \bigl(\operatorname{supp}(\gamma)\cap W\bigr)=\{z\}$. Then $\sigma$ must pass through~$z$ infinitely many times; since $\sigma$ is simple we deduce that $\sigma$ is a closed geodesic.
    
Assume now $\sigma$ is not a closed geodesic. Then there are no isolated point in $\gamma|_{[-\delta/2,\delta/2]}\cap W$; to complete the proof that $W$ is a transversally Cantor-like geodesic set we must prove that $\gamma|_{[-\delta/2,\delta/2]}\cap W$ is totally disconnected. Indeed, if not, it must contain a closed interval.
Then, since $\sigma$ is transversal to~$\gamma$, we can apply \cite[Proposition~4.1]{AB} to show that the $\omega$-limit set $W$ has not empty interior, against our assumption.
\end{proof}

Thus the $\omega$-limit set $W$ of a non-closed geodesic is a transversally Cantor-like geodesic set when $W$ has empty interior and contains $\operatorname{supp}(\sigma)$, a possiblity not considered in the proofs of \cite[Theorem 4.3]{AB} and \cite[Theorem 0.1]{AT1}. An example of this phenomenon can be found in \cite{DFG}.

\end{document}